\documentclass{amsart}
\usepackage{etoolbox,xcolor,cancel}
\patchcmd{\section}{\scshape}{\bfseries}{}{}
\makeatletter
\renewcommand{\@secnumfont}{\bfseries}
\makeatother
\calclayout
\usepackage{amsmath,amssymb,amscd,amsthm,enumitem}
\usepackage{graphicx,epstopdf}
\usepackage[titletoc,toc,title]{appendix}
\usepackage{color,soul}
\usepackage{float}
\usepackage{hyperref}
\usepackage[ruled,vlined,linesnumbered,algosection]{algorithm2e}
\usepackage{constants}
\newconstantfamily{B}{symbol=B}

\def\e{\epsilon}

\def\O{\Omega}

\def\bp{\begin{proposition}}
\def\ep{\end{proposition}}
\def\bt{\begin{theo}}
\def\et{\end{theo}}
\def\be{\begin{equation}}
\def\ee{\end{equation}}
\def\bl{\begin{lemma}}
\def\el{\end{lemma}}
\def\bc{\begin{corollary}}
\def\ec{\end{corollary}}
\def\pr{\noindent{\bf Proof: }}

\def\bd{\begin{definition}}
\def\ed{\end{definition}}
\def\CC{{\mathbb C}}

\def\R{{\mathbb R}}

\def\barP{{\cal \bar{P}}}
\def\P{{\cal P}}
\def\barE{\bar{E}_{\epsilon}(F)}
\def\F{{\cal F}}

\newcommand{\x}[1]{{}$\kern-2\mathsurround${}\binoppenalty10000 \relpenalty10000 #1{}$\kern-2\mathsurround${}}
\newcommand{\tab}{\hspace*{2em}}

\newcommand{\srf}{\ensuremath{SRF}} 
\DeclareMathOperator*{\argmin}{arg\,min}

\makeatletter
\DeclareRobustCommand*\cal{\@fontswitch\relax\mathcal}
\makeatother

\newtheorem{theo}{Theorem}[section]
\newtheorem{lemma}{Lemma}[section]
\newtheorem{definition}{Definition}[section]
\newtheorem{corollary}{Corollary}[section]
\newtheorem{proposition}{Proposition}[section]
\newtheorem{remark}{Remark}[section]
\newtheorem*{ThmInvStatement2}{Theorem \ref{thm:inv.fn}, statement 2}

\numberwithin{equation}{section}
\begin{document}

\title[Errors in solving Prony system]{Geometry of error amplification in solving Prony system with
near-colliding nodes}

\author[A. Akinshin]{Andrey Akinshin}
\address{Department of Mathematics,     The Weizmann Institute of Science, Rehovot 76100, Israel}
\address{Laboratory of Inverse Problems of Mathematical Physics, Sobolev Institute of Mathematics SB RAS, Novosibirsk 630090, Russia}
\email{andrey.akinshin@weizmann.ac.il}
\thanks{}

\author[G. Goldman]{Gil Goldman}
\address{Department of Mathematics,     The Weizmann Institute of Science, Rehovot 76100, Israel}
\email{gil.goldman@weizmann.ac.il}
\thanks{}

\author[Y. Yomdin]{Yosef Yomdin}
\address{Department of Mathematics,     The Weizmann Institute of Science, Rehovot 76100, Israel}
\email{yosef.yomdin@weizmann.ac.il}


\date{}

\begin{abstract}
We consider a reconstruction problem for ``spike-train'' signals $F$ of an a priori known form
$ F(x)=\sum_{j=1}^{d}a_{j}\delta\left(x-x_{j}\right),$ from their moments $m_k(F)=\int x^kF(x)dx.$
We assume that the moments $m_k(F)$, $k=0,1,\ldots,2d-1$, are known with an absolute error not exceeding $\e > 0$.
This problem is essentially equivalent to solving the Prony system $\sum_{j=1}^d a_jx_j^k=m_k(F), \ k=0,1,\ldots,2d-1.$

We study the ``geometry of error amplification'' in reconstruction of $F$ from $m_k(F),$
in situations where the nodes $x_1,\ldots,x_d$ near-collide, i.e. form a cluster of size $h \ll 1$.
We show that in this case, error amplification is governed by certain algebraic varieties in the parameter space of
signals $F$, which we call the ``Prony varieties''.

Based on this we produce lower and upper bounds, of the same order, on the worst case reconstruction error.
In addition we derive separate lower and upper bounds on the reconstruction of the amplitudes and the nodes.

Finally we discuss how to use the geometry of the Prony varieties to improve the reconstruction accuracy given additional a
priori information.

\end{abstract}

\maketitle

\section{Introduction}\label{Sec:Intro}

The problem of reconstruction of spike-trains, and of similar signals,
from noisy moment measurements, and a closely related problem of robust solving the classical Prony system,
is a prominent problem in Mathematics and Engineering. Here,
we consider the case when the nodes nearly collide, which is well known to present major mathematical difficulties,
and is closely related to a spike-train super resolution problem
(see
\cite{donoho1992superresolution}, \cite{lindberg_mathematical_2012} as a small sample).

%
%
\smallskip

The aim of this paper is to describe the patterns of amplification of the measurements error $\e$,
in the reconstruction process, caused by the geometric nature of the Prony system, independently
of the specific method of its inversion. We concentrate, following the line of \cite{akinshin2015accuracy,akinshin2017accuracy},
on the ``simplest non-trivial case'', where the nodes of a spike-train signal $F$ form a single cluster of size $h \ll 1$, with $d$ nodes, 
while the measurements are the first $2d$ real moments of $F$.


\subsection{Setting of the problem}\label{Sec:setting}

Assume that our signal $F(x)$ is a spike-train, i.e. a linear combination of $d$ shifted $\delta$-functions:
\be \label{eq:equation.model.delta}
F(x)=\sum_{j=1}^{d}a_{j}\delta\left(x-x_{j}\right),
\ee
where $a=(a_1,\ldots,a_d) \in {\mathbb R}^d, \ x=(x_1,\ldots,x_d) \in {\mathbb R}^d.$
We assume that the form (\ref{eq:equation.model.delta}) is a priori known,
but the specific parameters $(a,x)$ are unknown. Our goal is to reconstruct $(a,x)$
from $2d$ moments $m_k(F)=\int_{-\infty}^\infty x^k F(x)dx, \ k=0,\ldots,2d-1$,
which are known with a possible error bounded by $\e>0$.

\smallskip

An immediate computation shows that the moments $m_k(F)$ are
expressed via the unknown parameters $(a,x)$ as $m_k(F)=\sum_{j=1}^d a_j x_j^k$.
Hence our reconstruction problem is equivalent to solving the {\it Prony system} of algebraic equations, with the unknowns $a_j,x_j$:

\begin{align}\label{eq:Prony.system1}
	\sum_{j=1}^d a_j x_j^k &= \mu'_k,& |\mu'_k - m_k(F)| &\le \e,& k= 0,1,\ldots,2d-1.
\end{align}

This system and its various extensions and generalizations appears in many theoretical and applied problems (see
\cite{prony1795essai,auton1981investigation,stoica_spectral_2005,pereyra_exponential_2010,peter2013generalized,plonka2014prony,vetterli2002sampling}
and references therein).

We shall denote by ${\cal P}={\cal P}_d\subset \R^d \times \R^d$ the parameter space of signals $F$,
$$
	{\cal P}_d=\left\{(a,x): a=(a_1,\ldots,a_d)\in \R^d, x=(x_1,\ldots,x_d)\in \R^d, x_1<\ldots<x_d, a_j \ne 0
	\mbox{ for } j=1,\ldots,d \right\}, $$
and by ${\cal M}={\cal M}_{d} \cong {\mathbb R}^{2d}$ the moment space,
consisting of the $2d$-tuples of the moments $(m_0,m_1,\allowbreak\ldots,m_{2d-1})$.
We will identify signals $F$ with their parameters $(a,x)\in {\cal P}.$

\begin{remark}\label{rem.pd}
	The signals $F(x)=\sum_{j=1}^{d}a_{j}\delta\left(x-x_{j}\right)$, (considered as atomic measures, or as
	distributions), are symmetric with respect to the pairs $(a_j,x_j)$.
	In order to preserve uniqueness, we restrict our considerations to the nodes in the ``open ordered
	simplex'' $\{x_1<x_2<\ldots<x_d \} \subset {\mathbb R}^{d}$. 
\end{remark}

For any vector $v \in \R^{d}$, we denote by $\|v\|$ the maximum norm of $v$,
$$\|v\|=\max_j |v_j|.$$ 
Throughout this text we will always use the maximum norm $\|\cdot\|$ on ${\cal M}$ and on ${\cal P}$,
where for $F=(a,x)\in{\cal P}$,
$$
	\|F\|=\max (\|a\|,\|x\|).
$$

Let a signal $F$ as above be fixed. The main object we study in this paper is
the $\e$-error set $E_\e(F)$ consisting of all signals $F'(x)$ which may appear
as the reconstruction of $F,$ from the noisy moment measurements $\mu'_k$ with $|\mu'_k - m_k(F)|\le \e, \ k=0,\ldots,2d-1.$

\bd\label{def:error.set}
The error set $E_\e(F)\subset {\cal P}$ is the set consisting of all the signals $F'\in {\cal P}_d$ with
$$
|m_k(F')-m_k(F)|\le \e, \ k=0,\ldots,2d-1.
$$
\ed

Our ultimate goal is a detailed understanding of the geometry of the error set $E_\e(F)$,
in the various cases where the nodes of $F$ near-collide, and applying this information in order to improve the
reconstruction accuracy. The results presented here describe the geometry of the error set of a single cluster,
which, we will show, have very different scales of magnitude along certain algebraic curves. For this purpose
consider the following definition of a cluster configuration.

For a signal $F=(a,x)\in {\cal P}_d$ we denote by $I_F=[x_1,x_d]$ the minimal interval in
$\mathbb R$ containing all the nodes $x_1,\ldots,x_d$. We put $h(F)=\frac{1}{2}(x_d-x_1)$ to be the half of the length of $I_F$,
and put $\kappa(F)=\frac{1}{2}(x_1+x_d)$ to be the central point of $I_F$.
\bd[Regular cluster]
	For $F=(a,x) \in {\cal P}_d$ with $h = h(F) \le \frac{1}{2}, \kappa =\kappa(F)$,  $0 < m \le M$
	and $\eta>0$, we say that
	$F$ forms an $(h,\kappa,\eta,m,M)$-regular cluster if
	its amplitudes satisfy
		$$m\leq |a_j|\leq M, \tab j=1,\ldots,d,$$
	and the distance between any neighboring nodes
	$x_j,x_{j+1}, \ j=1,\ldots,d-1,$ is at least $\eta h$.
\ed

\smallskip

Now we define the ``Prony varieties'', which are just the coordinate subspaces of different dimensions,
with respect to the moment coordinates.
\bd\label{def:Prony.varieties}
For each $q=0,\ldots, 2d-1,$ and $\mu=(\mu_0,\ldots,\mu_q)\in \R^{q+1}$, the Prony variety $S_q(\mu)$
is an algebraic variety in $\R^{d}\times \R^{d}$,
defined as the set of all $(a,x) \in \R^{d}\times \R^{d}$ satisfying
\be\label{eq:Prony.system22}
\sum_{j=1}^d a_j x_j^k = \mu_k, \ k= 0,1,\ldots,q.
\ee
For a signal $F \in {\cal P}_d$ and $\mu = (m_0(F),m_1(F),\ldots,m_q(F))$ we will denote by $S_q(F)$ the variety
$S_q(\mu)$.
\ed

For a fixed signal $F$ and decreasing $q$ the Prony varieties $S_q(F)$ form an increasing chain of algebraic varieties:
$$
	S_{2d-1}(F)\subset S_{2d-2}(F)\subset \ldots \subset S_1(F)\subset S_0(F) \subset \R^{d}\times \R^{d}.
$$

Let us stress here some points.
Below we describe the geometry of the error set $E_\e(F)$ for a signal $F \in {\cal P}_d$ forming 
a regular cluster.
We will show a tight connection between this geometry and the Prony varieties $S_q(F)$, 
in a sufficiently small neighborhood $Q\subset \P_d$ of $F$. In fact, 
we require $Q$ to be contained in the area of validity of 
the Quantitative Inverse Function Theorem 
(which we state in Section \ref{Sec:Prony.mapping} and prove in the Addendum). 
Accordingly, we keep $\e$ small enough so that $E_\e(F)$ is contained in $Q$. 
Thus, we will always be interested only in the part $S_q(F) \cap Q \subset \P_d$ of the Prony variety 
$S_q(F)$.

We show in Section \ref{Sec:Prony.mapping} (although not stating it explicitly) that 
the intersection $S_q(F) \cap \P_d$ contains only regular points of $S_q(F)$, or, in other words, 
singularities of $S_q(F)$ occur only at the zero locus of the polynomial 
$\prod_{i<j}(x_i-x_j)\prod_{k}a_k.$ 

Notice that by definition of $\P_d$, the intersection $S_q(F) \cap \P_d$ contains 
no non-identical permutation of $F$ (see Remark \ref{rem.pd} above).  
Consequently we have that $S_{2d-1}(F) \cap Q =F$  (i.e. $S_{2d-1}(F)
\cap Q$ is the only solution of the Prony system \eqref{eq:Prony.system1} in ${\cal P}_d$),
while for $q=2d-2,\ldots,0$, $S_q(F) \cap Q$ are increasing regular subsets of $S_q(F)$ containing $F$.
In particular $S_{2d-2}(F)\cap Q$, which is important in the results below,
is a regular curve passing through $F$.

\begin{remark}
	Let $\Delta = \{x=(x_1,\ldots,x_d)\in {\mathbb R}^{d}, \ x_i=x_j \text \ {for \ some} \ i\ne j \}$ 
	be the generalized diagonal in ${\mathbb R}^{d}$. 
	In the present paper, starting with $F\in {\cal P}_d$, we always consider a neighborhood of $F$
	which is entirely in ${\cal P}_d$. In other words, such neighborhood never touches the ``node diagonal'' ${\mathbb
	R}^{d}\times \Delta$. Including into consideration the diagonal ${\mathbb R}^{d}\times \Delta$ is
	important in understanding the geometry of the Prony system, but it requires additional tools. 
	For some initial results in this direction see \cite{batenkov2013geometry,Prony.Sing.View.18,goldman2018geometry}. 
\end{remark}

\begin{remark} 		
	Let us mention here that the appearance of Algebraic Geometry in study of Prony systems is certainly not new 
	(compare \cite{kunis2016multidim.Prony,kunis2017multidim.bounds} and references therein in case 
	of multi-dimensional Prony systems). 
\end{remark}

\subsection{Normalization}
	Consider the following ``normalization'' applied on signals $F$ forming an $(h,\kappa,\eta,m,M)$-regular cluster:
	shifting the interval $I_F$ to have its center at the origin, and then rescaling $I_F$ to the interval $[-1,1]$.
	For this purpose we define, for each $\kappa \in {\mathbb R}$ and $h>0$ the transformation
	\be\label{eq:Ptransform}
	\Psi_{\kappa,h}:{\cal P}_d\to {\cal P}_d,
	\ee
	defined by $(a,x)\to (a,\bar x),$ with
	$$
	\bar x=(\bar{x}_1,\ldots,\bar{x}_d), \ \ \bar{x}_j=\frac{1}{h}\left (x_j-\kappa\right), \ j=1,\ldots,d.
	$$
	For a given signal $F$ with $h=h(F)$ and $\kappa=\kappa(F)$, we call the signal $G=\Psi_{\kappa,h}(F)$ the model signal
	of $F$.
	Clearly, $h(G)=1$ and $\kappa(G)=0$. Explicitly $G$ is written as
	$$
	G(x)=\sum_{j=1}^{d}a_{j}\delta\left(x-\bar{x}_{j}\right).
	$$
	With a certain misuse of notation, we will denote the space ${\cal P}_d$ containing the model signals $G$ by
	$\bar {\cal P}_d,$ and call it ``the model space''.
	
	\smallskip
	
	For a given $F\in {\cal P}_d$ with the model signal $G=\Psi_{\kappa,h}(F),$
	we denote by $\bar E_\e(F)$ the ``normalized" error set:
	$$\bar E_\e(F) = \Psi_{\kappa,h}(E_\e(F)).$$
	
	\smallskip
	
	Let $F$ form a cluster of size $h\ll 1,$
	while inside the cluster the nodes are well separated from one another.
	The reason for mapping such signal $F$ into the model space is that in this case we will show that
	the moment coordinates centered at $F$,
	$$(m_0(F')-m_0(F),\ldots,\allowbreak  m_{2d-1}(F')-m_{2d-1}(F)),$$
	are ``stretched'' in
	some directions, up to an order of $(\frac{1}{h})^{2d-1}$.
	{\it In contrast,
	the coordinates system
	$$(m_0(G')-m_0(G),\ldots, m_{2d-1}(G')-m_{2d-1}(G))$$
	is bi-Lipschitz
	equivalent to the standard coordinates $(a,\bar x)$ of $\bar {\cal P}_d$, in a neighborhood of order $h^{2d-1}$ around $G$}
	(quantitative inverse function theorem, see Theorem \ref{thm:inv.fn} below).

	\smallskip
	
	Below we describe the geometry of the error set $\bar E_\e(F)$ in the associated model space $\bar {\cal P}_d$.
	{\it Note that $\bar E_\e(F)$ is simply a translated and rescaled version of $E_\e(F)$ in the nodes coordinates.}
	Hence, the description of $\bar E_\e(F)$ directly describes $E_\e(F)$ via the inverse transformations.

\subsection{Sketch of the main results}\label{Sec:setting.sketch}

Let the nodes $x_1,\ldots,x_d$ of $F$ form a cluster of size $h\ll 1$ and
let $G=\Psi_{\kappa,h}(F)$ be the model signal of $F$.
Informally, our main results in the case  $\e$ of order $h^{2d-1}$ or less are
the following:

\smallskip

\begin{enumerate}[leftmargin=*]
  \item
	    {\it
	    	In Section \ref{Sec:error.set.geometry}, Theorem \ref{thm.error.set.geometry.no.shift} and Theorem \ref{thm.error.set.geometry.shift}, 
	    	we describe the geometry of the error set $\bar{E}_\e(F)$. It is shown that the Prony Varieties provide the ``principal components'' of the error set
	    	in the following sense:
	    	For each $q=2d-1,\ldots,0$, $\bar E_\e(F)$ is contained within a neighborhood of size
	    	$\sim \left(\frac{1}{h}\right)^{q}\e$ of the Prony variety $S_q(G)$. Put differently,
	    	the width of $\bar{E}_\e(F)$ in the direction of the model moment coordinate $m_k(G')-m_{k}(G)$,
	    	$k=0,\ldots,2d-1,$ is of order $ h^{-k} \e.$ See Figures \ref{fig.h01} and \ref{fig.h05}
			below.
	    	}
	    \begin{figure}
			\centering
			\includegraphics[scale=0.75]{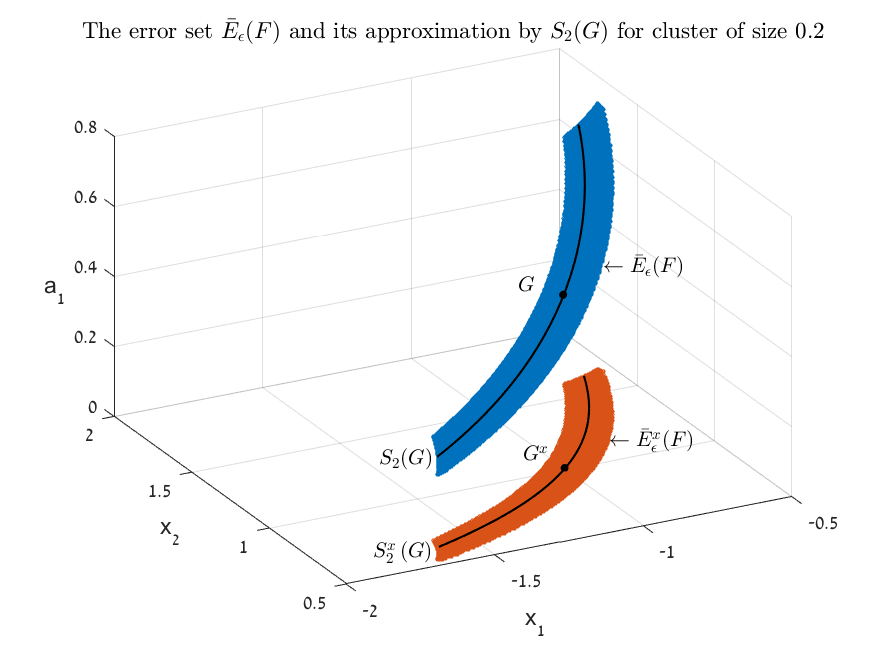}
			\caption{The projections of the error set $\bar{E}_{\e}(F)$ and a section of the Prony
			curve $S_2(G)$, for $F= \frac{1}{2} \delta\left(x+0.1 \right) + \frac{1}{2} \delta\left(x-0.1 \right)$,
			$h=0.1$, $\epsilon = h^3$ and $G=\Psi_{0,h}(F)$.}
			\label{fig.h01}
			\includegraphics[scale=0.75]{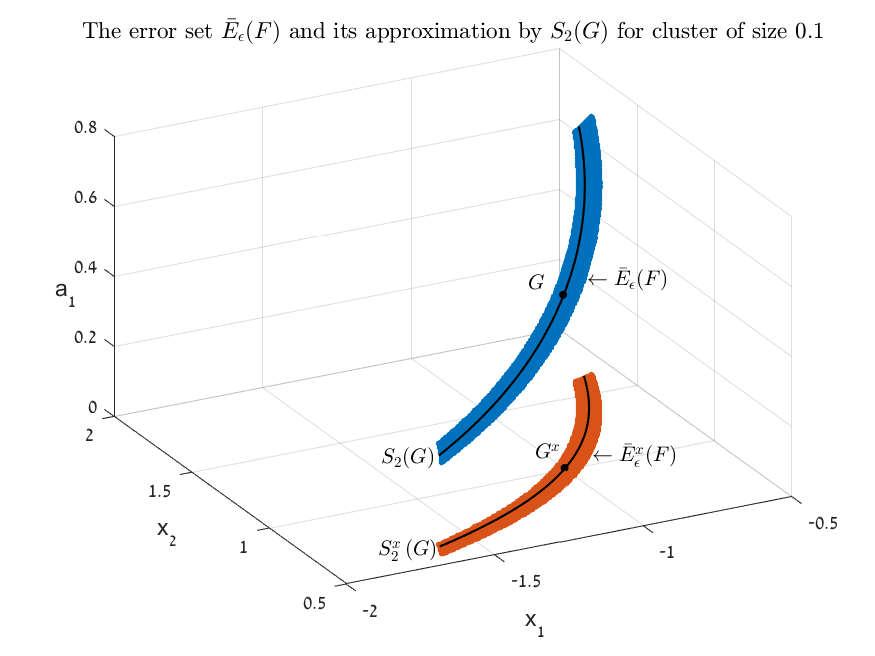}
			\caption{The projections of the error set $\bar{E}_{\e}(F)$ and a section of the Prony
			curve $S_2(G)$, for $F= \frac{1}{2} \delta\left(x+0.05 \right) + \frac{1}{2} \delta\left(x-0.05 \right)$,
			$h=0.05$, $\epsilon = h^3$ and $G=\Psi_{0,h}(F)$. Note the convergence of
			$\bar{E}_{\e}(F)$ into $S_2(G)$ as the cluster size reduces.}
			\label{fig.h05}
		\end{figure}
   \item\label{item.sketch.worst.case.error}	
	   	{\it In Section \ref{sec:worst.case.reconstruction.error}, Theorems \ref{thm.upper.bound}
	   	and \ref{thm.lower.bound}, we use
		the above result to derive lower and upper bounds, of the same order, on the worst case reconstruction error.
	   	We show that:\\
	   	The worst case reconstruction error,
	   	$$\rho(F,\e)= \max_{F'\in E_\e(F)}\|F'-F\|,$$
	   	is of order $\e h^{-2d+1}$.\\
	   	The worst case reconstruction error of the amplitudes,
	   	$$\rho_a(F,\e)= \max_{F'=(a',x')\in E_\e(F)}\|a'-a\|,$$
	   	is of order $\e h^{-2d+1}$. \\
	   	The worst case reconstruction error of the nodes,
	   	$$\rho_x(F,\e)= \max_{F'=(a',x')\in E_\e(F)}\|x'-x\|,$$
	   	is of order $\e h^{-2d+2}$.
	   	
	   	We stress that reconstructions $F'$ with reconstruction errors as above cannot occur
	   	everywhere: they fall into a small neighborhood of the Prony curve
		$S_{2d-2}(F)$. This fact is used in Section \ref{sec:improving.accuracy} to improve the reconstruction accuracy (see item \ref{point.4} below).}
		
		Our next result concerns the accuracy of reconstruction of the Prony varieties $S_q(F)$:	
	\item\label{item.sketch.worst.case.error.prony.varieties}
		{\it While the point worst case reconstruction error of the signal $F$ is of order $\e h^{-2d+1}$, the curve $S_{2d-2}(F)$ itself can be
		reconstructed with a better accuracy of order $\e h^{-(2d-2)}$. The ``hierarchy of the accuracy rates'' is continued
		along the chain $S_{2d-1}(F)\subset \ldots \subset S_0(F)$ of the Prony varieties $S_q(F)$: each $S_q(F)$
		can be reconstructed with an accuracy of order $\e h^{-q}$. See Theorem \ref{thm.prony.reconstruction}.}
		\footnote{Through this text we assume that the
		Prony inversion (when possible) is accurate, and that the reconstruction error is caused only by the measurements error.
		Moreover, we will always assume below that all the ``algebraic-geometric'' operations, with the known parameters, are performed accurately.
		Specifically this concerns constructing certain algebraic curves and higher-dimensional varieties.
		Of course, such algorithmic constructions in Computational Algebraic Geometry may present well-known difficulties,
		but in the present paper we do not touch this topic.}

		Based on the theory developed in sections \ref{Sec:error.set.geometry} and \ref{sec:worst.case.reconstruction.error}
		we conclude our results with the following
		fact:
	\item\label{point.4}
		{\it If a certain additional a priori information is available on the signal $F$,
		the reconstruction accuracy can be significantly improved via the following procedure:
		first we reconstruct the Prony variety $S_q(F)$ for a certain appropriate $q$.
		The accuracy of this reconstruction (of order $\e h^{-q}$) is higher than that of a single point solution.
		Then we use the additional information available in order to accurately localize the signal $F$ inside the Prony variety $S_q(F)$.
		In section \ref{sec:improving.accuracy} we demonstrate this procedure and how it can improve the reconstruction accuracy with
		respect to Prony method.}
\end{enumerate}

\begin{remark}
	Consider the case $\e$ of order greater than $h^{2d-1}$.
	Our approach, based on the regularity of the moment coordinates,
	does not apply here since for large errors the reconstruction encounters singularities.
	We do not study this case here, however, the Prony varieties $S_q$, being algebraic objects that are defined globally, remain a relevant tool in studying error amplification and
	collision singularities in much larger scales (See \cite{batenkov2013geometry,Prony.Sing.View.18,goldman2018geometry}).
\end{remark}

\subsection{Organization of the text}
	
In Section \ref{sec:related.work} we discuss related settings and results. 
In particular, we explain in detail the connection between the results of the resent paper
to the case of Fourier measurements / super resolution setting (in particular, with $N\gg 2d$ or continuous samples), 
and the possible extensions to the case of several clusters.

In Section \ref{sec:improving.accuracy} we show possible applications of our results to improving the
reconstruction accuracy of Prony method.
We provide a simple example, supported by numerical simulations,
where taking into account the Prony varieties,
significantly improves the reconstruction accuracy.

\medskip

Sections \ref{Sec:Prony.mapping} - \ref{sec:worst.case.reconstruction.error} are devoted to the accurate stating of the
results and their proofs.
In Section \ref{Sec:Prony.mapping} we introduce the ``Prony mapping'', and study its inversion via ``Quantitative inverse function theorem''. In
Section \ref{Sec:error.set.geometry} our main results on the geometry of the error set are stated and proved. 
In Section \ref{sec:worst.case.reconstruction.error} we derive, based on the previous section, 
tight estimates on the worst case reconstruction error. Finally, in Appendix
(\ref{appendix.Quantitative}), we proof a specific form of the quantitative inverse function theorem, 
giving explicit expression for the constants used in the text.

\subsection{Acknowledgements}
The research of GG and YY is supported in part by the Minerva
Foundation. The authors would like to thank the referees for suggesting significant improvements
in the presentation. 

\section{Related work and discussion}\label{sec:related.work}

As it was already mentioned in the Introduction, in the present paper we concentrate on a rather restricted case of the spike-train reconstruction problem. 
First, we take the real moments as the measurements (instead of much more common and natural Fourier samples). 
Second, we take exactly $2d$ moment measurements (instead of $N \gg 2d$ moments or Fourier samples). 
Finally, we assume that the nodes of $F$ form exactly one cluster, instead of the more general configuration of several clusters.

The main reasons for us to insist on this setting 
is that it presents in a relatively compact form
the most essential patterns of the error amplification in multi-cluster moment  / Fourier spike-train reconstruction.
We discuss this fact in detail in subsections \ref{sec.super.res}, 
\ref{more.measurments}, \ref{sec.several.clusters} below. 

\subsection{Clustered Fourier reconstruction (super-resolution)}\label{sec.super.res}

In this section we outline the tight connection between the super-resolution problem, 
where the measurements are Fourier samples, and the results of the present paper about moment reconstruction. 
In fact, up to constants, the error set in the case of Fourier measurements {\it is described by exactly the same moment
inequalities, as in the present paper}.
	
\smallskip
	
	For a signal $F$ of the form \eqref{eq:equation.model.delta}, let ${\cal{F}}\left(F\right)$ denote the Fourier transform of $F$:
	$$
		{\cal F}(F)(s) = \int_{-\infty}^\infty F(x)e^{-2 \pi i x s }dx = \sum_{j=1}^d a_j e^{-2\pi i x_j s }.
	$$
	
	In a super resolution setting, it is frequently assumed that the measurements for the reconstruction of $F$
	are given as a function $\Phi$ satisfying
	\begin{equation}\label{eq:noise-condition}
	  \left| \Phi(s) - {\cal F}(F)(s) \right| \le \epsilon,\quad s\in[-\Omega,\Omega].
	\end{equation}
	where $\epsilon>0$ is the noise level and $\Omega>0$ is the band limit.
	
	\smallskip
	
	Similarly to the moment $\epsilon$-error set \ref{def:error.set}, we define the Fourier $\epsilon$-error set
	as follows.
	
	\bd\label{def.error.set.Fourier} For $\e,\Omega >0$ and $F \in {\cal P}_d$, the Fourier error set
  		$E_{\epsilon,\Omega}(F) \subset {\cal P}_d$ is the set
  		consisting of all the signals $F'\in {\cal P}_d$ with
	  	\begin{align*}
	  		\left|\F(F')(s)-\F(F)(s) \right|\le \epsilon, & &
	        s\in[-\Omega,\Omega].
	  	\end{align*}
	\ed	

	Let $F$ form an $(h,\kappa,\eta,m,M)$-regular cluster
	as the case considered in this paper. Define the super resolution factor as
	$$\srf = \frac{1}{\Omega h}.$$
	The radius of the Fourier error set, or equivalently the worst case reconstruction error of $F$,
	in the super resolution setting \eqref{eq:noise-condition}, was shown to scale like
	$\srf^{2d-1}\epsilon$ (see \cite{akinshin2015accuracy,superres_clusters18,li2019super} for off-grid setting and 
	\cite{donoho1992superresolution, demanet2015recoverability, batenkov2018conditioning,li2017stable}
	for on-grid setting). If we further assume that at most $l\le d$ nodes of $F$ form a cluster of size $h$,
	then recent results show that the scaling of the radius of the error set improves to an order of
	$\srf^{2l-1}\epsilon$ (see \cite{batenkov2018conditioning,li2017stable,kunis2018condition,superres_clusters18,li2019super,kunis2019smallest}).
	
	The Fourier error set and the moment error set are related via the Taylor series
	expansion of the Fourier transform, that is expressed using 
	the moments as follows (see \cite[Proposition
	3.1]{akinshin2015accuracy}):
	\be \label{eq.fourier.taylor}
		{\cal F}(F)(s)=\sum_{k=0}^\infty {\frac{m_k(F)}{k!}}\tilde s^k, \ \text where \ \ \tilde s = -2\pi i s.
	\ee
	In fact it is possible to show that {\it these sets are equivalent} in the following sense:
	
	Let $F=(a,x) \in {\cal P}_d$ form an $(h,\kappa,\eta,m,M)$-regular cluster.
	Then, there exist positive constants $\Cl[B]{srf.low.equi},\Cl[B]{noise.up.equi}$ and $\Cl[B]{low.equi} \le 1 \le
	\Cl[B]{high.equi}$, depending only on $\eta, d, m,$ such that for each $\Omega, \epsilon$ satisfying $\srf \ge
	\Cr{srf.low.equi}$ and $0 \le \epsilon\le \Cr{noise.up.equi} (\srf)^{-2d+1}$, it holds that
	\begin{equation}\label{eq.fourier.moment.error.set.relation.1}
			 E_{\Cr{low.equi}\epsilon}(\Psi_{\kappa,\frac{1}{\Omega}}(F)) \subseteq
		     \Psi_{\kappa,\frac{1}{\Omega}}\big(E_{\epsilon,\Omega}(F)\big)\subseteq
		     E_{\Cr{high.equi} \epsilon}(\Psi_{\kappa,\frac{1}{\Omega}}(F)),
	\end{equation}	
	or equivalently
	\begin{equation}\label{eq.fourier.moment.error.set.relation.2}
			 \Psi^{-1}_{\kappa,\frac{1}{\Omega}}\big(E_{\Cr{low.equi}\epsilon}(\Psi_{\kappa,\frac{1}{\Omega}}(F))\big)	 \subseteq
		     E_{\epsilon,\Omega}(F)\subseteq
		     \Psi^{-1}_{\kappa,\frac{1}{\Omega}}\big(E_{\Cr{high.equi} \epsilon}(\Psi_{\kappa,\frac{1}{\Omega}}(F))\big).
	\end{equation}
	
	\smallskip
	
	Put differently, for a signal $F$ with clustered nodes as above, and for any signal $F_1 \in {\cal P}_d$, 
	the Fourier difference is $\epsilon$ small, i.e.
	$$\max_{s \in [-\Omega, \Omega]}|{\cal F}(F_1)(s) - {\cal F}(F)(s)|\le \epsilon,$$
	if and only if the moments $m_0,\ldots,m_{2d-1}$, of the centered and {\it scaled by $\Omega$}
	difference signal $\Psi_{\kappa,\frac{1}{\Omega}}(F) - \Psi_{\kappa,\frac{1}{\Omega}}(F_1)$, are
	order of $\epsilon$ small. 
	
	The main result of this paper concerning the geometry of 
	moment reconstruction (see Theorem
	\ref{thm.error.set.geometry.no.shift} and Theorem \ref{thm.error.set.geometry.shift}) is extended to Fourier
	reconstruction via relation \eqref{eq.fourier.moment.error.set.relation.1} (or relation
	\eqref{eq.fourier.moment.error.set.relation.2}), as follows. See also Figure \ref{fig.h01.fourier}.
    \begin{corollary}\label{cor.fourier.error.set}
    	Let $F=(a,x) \in {\cal P}_d$ form an $(h,\kappa,\eta,m,M)$-regular cluster. 
    	Then, there exist positive constants $\Cl[B]{srf.low.cor.main},\Cl[B]{noise.cor.main},\Cl[B]{cor.main}$,
    	depending only on $d,\eta,m,M$, such that for each $\Omega, \epsilon$ satisfying 
    	$$\srf \ge \Cr{srf.low.cor.main}  \mbox{\tab and \tab } 0 \le \epsilon\le \Cr{noise.cor.main} (\srf)^{-2d+1},$$
    	$E_{\epsilon,\Omega}(F)$ is contained within the 
    	$\Delta_q$-neighborhood of the Prony variety $S_{q}(F)$, for
 		\begin{align*}
				\Delta_q&=\Cr{cor.main}\left(\srf\right)^{q}\e,& q=0,\ldots,2d-1. 		
 		\end{align*}
	\end{corollary}	
	See proof of Corollary \ref{cor.fourier.error.set} in Appendix \ref{appendix.moment.fourier}.

	\begin{figure}
		\centering
		\includegraphics[scale=0.77]{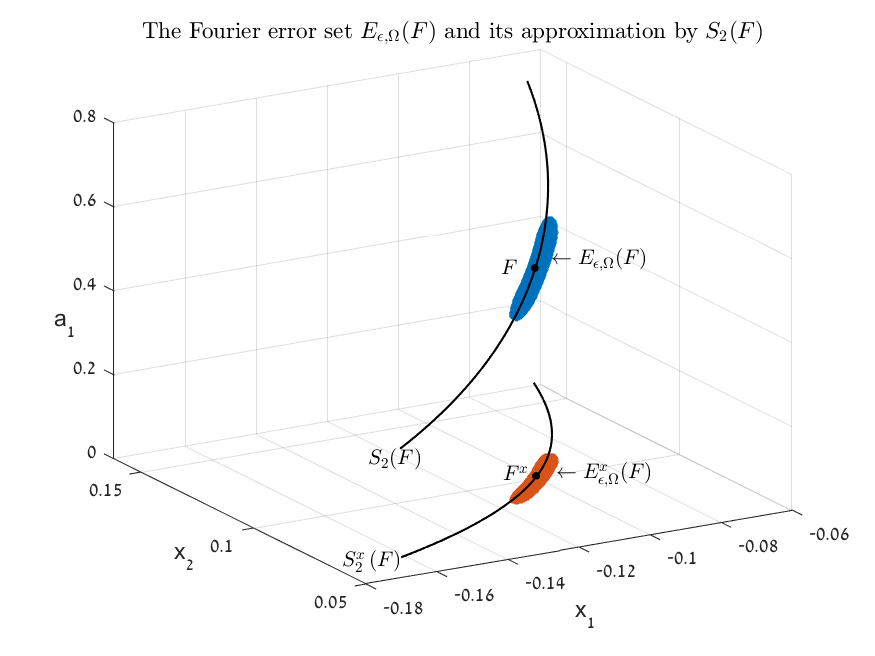}
		\caption{The projections of the Fourier error set $E_{\e,\Omega}(F)$ and a section of the Prony
		curve $S_2(F)$, for $F= \frac{1}{2} \delta\left(x+0.1 \right) + \frac{1}{2} \delta\left(x-0.1 \right)$,
		$h=0.1$ and $\epsilon = 5 h^3$. Compare with Figure \ref{fig.h01}.}
		\label{fig.h01.fourier}
	\end{figure}
	
	\smallskip
	
	This geometrical structure of the Fourier error set suggests a similar procedure to
	improve the reconstruction accuracy, as we demonstrate for the Prony method, in Section
	\ref{sec:improving.accuracy}. We intend to present results in this direction in future work.

	Finally, using the equivalence relation \eqref{eq.fourier.moment.error.set.relation.1}
	(or relation \eqref{eq.fourier.moment.error.set.relation.2}), 
	one can also derive the worst case reconstruction error rates for Fourier reconstruction of a signal cluster, 
	based on the corresponding moment reconstruction error
	rates that we prove in Section \ref{sec:worst.case.reconstruction.error}, 
	Theorem \ref{thm.upper.bound} and Theorem \ref{thm.lower.bound}.      
	
	\begin{corollary}\label{cor.fourier.worst}
		Let $F=(a,x) \in {\cal P}_d$ form an $(h,\kappa,\eta,m,M)$-regular cluster. 
		Then, there exist positive constants $\Cl[B]{cor.worst.srf},\Cl[B]{cor.worst.noise}$, depending only on $d,\eta,m,M$,
		such that for each $\Omega, \epsilon$ satisfying 
    	$$\srf \ge \Cr{cor.worst.srf}  \mbox{\tab and \tab } 0 \le \epsilon\le \Cr{cor.worst.noise} (\srf)^{-2d+1},$$
    	It holds that:
    	\begin{align}
    		\max_{F'\in E_{\e,\O}(F)}\|F'-F\| & \asymp \srf^{2d-1}\e,\\
    		\max_{F'=(a',x')\in E_{\e,\O}(F)}\|a'-a\| & \asymp \srf^{2d-1}\e,\\
    		\max_{F'=(a',x')\in E_{\e,\O}(F)}\|x'-x\| & \asymp \srf^{2d-1}h\e.  
    	\end{align}
	\end{corollary} 
	See proof of Corollary \ref{cor.fourier.worst} in Appendix \ref{appendix.moment.fourier}.
	
	\subsection{More measurements}\label{more.measurments}

In the present paper, we keep the number of moment measurements exactly $2d$: 
this is enough to obtain the correct error asymptotic behavior for the cluster size $h\ll 1$. 

However, the results of this paper can be used in order to accurately estimate
the worst case reconstruction error / minimax error rate in multi-cluster super resolution setting. 
This is done in \cite{superres_clusters18}, in the following main steps:

\smallskip

1. Let $F=\sum_{j=1}^d a_j\delta(x-x_j), \ a_j\in \CC, \ x_j \in {\mathbb R}$. We apply ``decimation'' (see \cite{batenkov2017accurate}), i.e. take exactly $2d$ uniformly spaced Fourier samples, with the step-size $\lambda$ of order $\frac{\O}{2d}$. In other words, we use ``most of the available bandwidth $\O$'', keeping the number of the samples $2d$. As a result we get a Prony system with the nodes $z_j=e^{2\pi i\lambda x_j}
$ on the unit circle. Clearly, the size $h$ of any cluster becomes $\lambda h \sim \O h$.

\smallskip

2. We show that for ``many'' values of $\lambda$ no new proximities between the nodes on the circle are created.

\smallskip

3. We apply the approach of the present paper (but with the ``quantitative inverse function theorem'' 
extended to the complex spaces), and finally produce the accuracy bounds of the required form, 
with $h$ replaced by $\O h$ (see Section \ref{sec.super.res}). 
This gives a ``correct'' decay rate of the reconstruction error, with respect to the bandwidth $\O$.

\smallskip

Available studies of certain high-resolution algorithms such as MUSIC
\cite{liao2016music}, ESPRIT/Matrix Pencil
\cite{fannjiang_compressive_2016}, Approximate Prony Method
\cite{potts2010bounds}, multivariate Prony method \cite{kunis2017multidim.bounds} and others provide rigorous
performance guarantees for the case $\srf<1$. 
We hope that our proof techniques here and in \cite{superres_clusters18} may be used in deriving the
stability limits of these and other methods in the super-resolution
regime, i.e. for $\srf>1$.  




\subsection{The case of several clusters}\label{sec.several.clusters}

Our description of the error set, via the moment inequalities, and of its ``skeleton'', 
provided by the hierarchy of the Prony varieties, extends to spike train signals forming several clusters. 
Let $F$ be a signal with the node clusters $Q_1,\ldots,Q_m$, each $Q_s$
being of size $h_s$ and containing $d_s$ nodes, $s=1,\ldots,m$. Denote by $F_s$ the ``local signals'', 
corresponding to the clusters $Q_s$. The main fact in this situation is the following:

\smallskip

{\it If the clusters $Q_s$ of $F$ are ``well-separated'', in comparison to their size, then the error set of $F$ is, essentially, the Cartesian product of the ``local'' error sets of $F_s, \ s=1,\ldots,m$. This up to constants, depending on the mutual position of the clusters $Q_s$, on their ``multiplicities'' $d_s$, and on their sizes $h_s$}.

\smallskip

This claim follows from the ``mutual independence'' of the local signals $F_s$,
corresponding to the node clusters $Q_s$:

\smallskip

{\it The errors in the moments of the local signals $F_s$ cannot cancel in the moments of their sum $F$}.

\smallskip

This last property is important in many questions far beyond the study of multi-cluster error sets and Prony varieties. 
Through the Jacobian of the Prony mapping it is closely connected with the properties of Vandermonde matrices with clustered nodes. 
Recently we've shown in \cite{batenkov2019spectral} that the column subspaces of the clusters of a rectangular 
Vandermonde matrix are near-orthogonal, for the parameters in a ``correct range''. 
This result strongly supports the ``mutual independence'' of the local signals $F_s$, corresponding to the node clusters $Q_s$.

\smallskip

Consequently, also the description of the error set using the Prony varieties, given in the present paper for one cluster,
extends to the multi-cluster case via the Cartesian products of the local Prony varieties as follows:
For each $q$, consider the subvariety $\tilde S_q$ in the signal space, which is the Cartesian product of the ``local'' Prony varieties $S_q^s$ corresponding to the clusters
$Q_s$:
$$
\tilde S_q=S^1_q\times S^2_q \times \ldots \times S^m_q.
$$
We see immediately that the moments up to $q$ are constant on $\tilde S_q$, 
while the higher moments $m_k$ can be locally bounded through the $k$-th powers of the cluster sizes $h_s$. 
Consequently, $\tilde S_q$ play in the multi-cluster case the same role of a ``skeleton'' of the error set, 
as $S_q$ in the case of one cluster, described in detail in the present paper.

\smallskip



\smallskip

Thus, in principle, the main results of the current paper can be extended to several clusters. However, technically, the accurate description becomes rather involved. Still, we believe that a detailed understanding of the ``algebraic-geometric skeleton'' of the error amplification in the case of several clusters is highly important. We plan to present results in this direction separately.

\section{Improving the reconstruction accuracy given some additional information}\label{sec:improving.accuracy}
In this section we shortly discuss the way one can use the
Prony varieties in order to improve the reconstruction accuracy
of a spike train signal from its $2d$ initial moments.
Specifically, we show that Prony varieties can help to optimally utilize an
additional information on the
reconstructed signals.

As we explain in Section \ref{sec.super.res},
the spreading and scale of the error in Fourier reconstruction is tightly
connected to moment reconstruction via \eqref{eq.fourier.moment.error.set.relation.2}
(see also Figure \ref{fig.h01.fourier}).
We therefore expect that the procedure we describe here can ultimately help to improve the accuracy of widely used
Fourier reconstruction methods - ESPRIT, APM, Matrix pencil and variants. We intend to present results in
this direction in future separate work.

\smallskip

Assume that the measured signal $F$, is known to form a small regular cluster of size $h\ll 1$.
Assume in addition that we have certain additional information on the signal $F$.
We do not specify here the nature of this information, which can either be known a priori or
a result of a different, non-moment, measurement of the signal,
assuming just that the measured signal {\it is known
to reside within a subset $\Omega \subset {\cal P}$.}

\smallskip

Recall that for measurement error $\e \ge 0 $, our input for the reconstruction of $F$ are the moment measurements
$\mu'=(\mu'_0,\ldots,\mu'_{2d-1})$ with
\begin{align}\label{eq:Prony.system2}
	 |\mu_k' - m_k(F)| \le \e, \tab k= 0,1,\ldots,2d-1.
\end{align}

Now consider the
following reconstruction procedure:

\smallskip

\begin{algorithm}
  \SetKwInOut{Input}{Input} \SetKwInOut{Output}{Output}
  \Input{number of nodes $d$.}
  \Input{measured moments $\mu'=(\mu'_0,\ldots,\mu'_{2d-1})$ satisfying \eqref{eq:Prony.system2}.}
  \Input{feasible set $\Omega \subset {\cal P}_d$.}
  \Output{an estimate $F^{PCRP} \in {\cal P}_d$.}
  Solve the Prony system \eqref{eq:Prony.system2} with input $\mu'$ and recover the signal $F'$\;
  Use $F'$ to reconstruct the Prony curve $S_{2d-1}(F')$\;
  \If{$S_{2d-1}(F') \cap \Omega \ne \emptyset$}{
  	 Find a signal $F^{PCRP}$ which is closest to $F'$ in the intersection $S_{2d-1}(F') \cap \Omega$, i.e.
  	 $$F^{PCRP} = \argmin_{F \in S_{2d-1}(F')\cap \Omega} \|F'-F\|;$$\label{step.standard.case}
  \noindent \Return the estimate $F^{PCRP}$.
  }\Else{
  	  Find a signal $F^{SRP}$ which is closest to $F'$ in the feasible set $\Omega$, i.e.
      $$F^{SRP} = \argmin_{F \in \Omega} \|F'-F\|;$$\label{step.edge.case}
      \Return the estimate $F^{SRP}$.
  }
  \caption{Prony curve reconstruction procedure given a priori information - PCRP}
  \label{alg:prony.curve}
\end{algorithm}
%
%
%
We compare the above procedure to the following ``natural'' solution algorithm using Prony method, which does not
relies on Prony curves (and appears as an edge case of the PCRP in step \ref{step.edge.case}):

\smallskip

\begin{algorithm}[hbt]
  \SetKwInOut{Input}{Input} \SetKwInOut{Output}{Output}
  \Input{number of nodes $d$.}
  \Input{measured moments $\mu'=(\mu'_0,\ldots,\mu'_{2d-1})$ satisfying \eqref{eq:Prony.system2}.}
  \Input{feasible set $\Omega \subset {\cal P}_d$.}
  \Output{an estimate $F^{SRP} \in {\cal P}_d$.}
  Solve the Prony system \eqref{eq:Prony.system2} with input $\mu'$ and recover the signal $F'$\;
  Find a signal $F^{SRP}$ which is closest to $F'$ in the feasible set $\Omega$, i.e.
  $$F^{SRP} = \argmin_{F \in \Omega} \|F'-F\|;$$
  \Return the estimate $F^{SRP}$.
  \caption{Standard reconstruction procedure given a priori information- SRP}
  \label{alg:prony}
\end{algorithm}
%
%
Let us now explain why the reconstruction procedure using the Prony curve, PRCP, is expected to improve
the accuracy with respect to standard reconstruction procedure, SRP, of solving the Prony system and then projecting the solution
into the feasible set.

\smallskip

Consider the solution $F'$ to the Prony system \eqref{eq:Prony.system2}, with input $\mu'$, appearing as
a first step in both reconstruction procedures. The distance of $F'$ from $F$, in the worst case, is of 
order $h^{-2d+1} \e$ (see item \ref{item.sketch.worst.case.error} in the sketch of the main results or 
the formal result in Theorem \ref{thm.lower.bound}). We have that the true solution $F$ is contained in an order of $h^{-2d+2}\e$
neighborhood of the Prony curve $S_{2d-1}(F')$ 
(see item \ref{item.sketch.worst.case.error.prony.varieties} in the sketch of the main results or the formal result in 
Corollary \ref{cor.upper.bound.prony.varieties.reconstruction}). 

At the final step of the SRP we take the closest signal to $F'$ in $\Omega$.
This closest point is typically at the same distance of order $h^{-2d+1} \e$ from $F$.
In contrast, in the PCRP we take the closest signal to $F'$ in $\tilde{\Omega}=S_{2d-1}(F')\cap\Omega$
(presuming that this set is non-empty, see step \ref{step.standard.case} in the PCRP).
Now since $F$ is located in a tiny belt around $S_{2d-1}(F')$, and provided that the diameter of $\tilde{\Omega}$
is of order $h^{-2d+2}\e$ or less,
we get an order of $h$-magnitude better accuracy guarantees compared to the SRP. That is, in such case we get that the worst case
reconstruction error of the PCRP is $\sim h^{-2d+2}\e$, while the worst case reconstruction error of the SRP is $\sim h^{-2d+1}\e$.

The same explanation as above holds for comparing the reconstruction accuracy of the nodes
of $F$, but with all accuracy bounds multiplied by $h$.
That is, if the diameter of the projection of
$\tilde{\Omega}$ into the nodes coordinates is $\sim h^{-2d+3}\e$, then the worst case
reconstruction error of the nodes using the PCRP is $\sim h^{-2d+3}\e$,
whereas the worst case reconstruction error using the SRP is $\sim h^{-2d+2}\e$.

\smallskip

In Figure \ref{fig.improved1} we demonstrate this effect on the reconstruction of the nodes of $F$.

\begin{figure}
		\centering
		\includegraphics[scale=0.73]{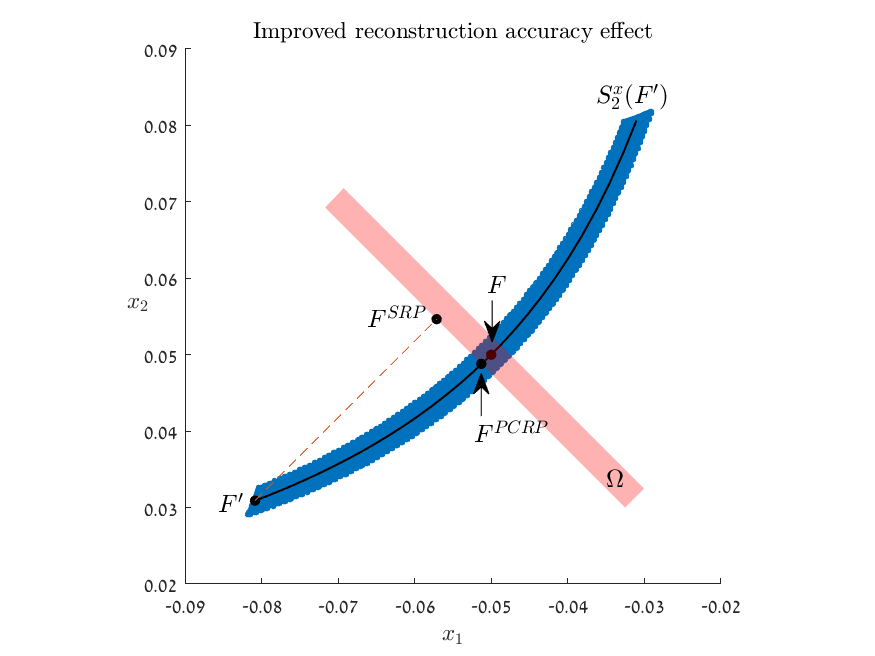}
 		\caption{Above $F=\frac{1}{2}\delta(x+0.05)+\frac{1}{2}\delta(x-0.05)$ and $\e = h^{3}$.
 		$F'$ is the reconstruction from the measurements $\mu'=(1,0,h^2,-h^3)$.
     	In blue is the $\e$-error set. Note the improved reconstruction
     	$F^{PCRP}$ attained by moving over the Prony curve, compared to $F^{SRP}$.}
		\label{fig.improved1}
\end{figure}

\subsection{Numerical experiments}

\begin{figure}[!htb]
		\centering
		\includegraphics[scale=0.69]{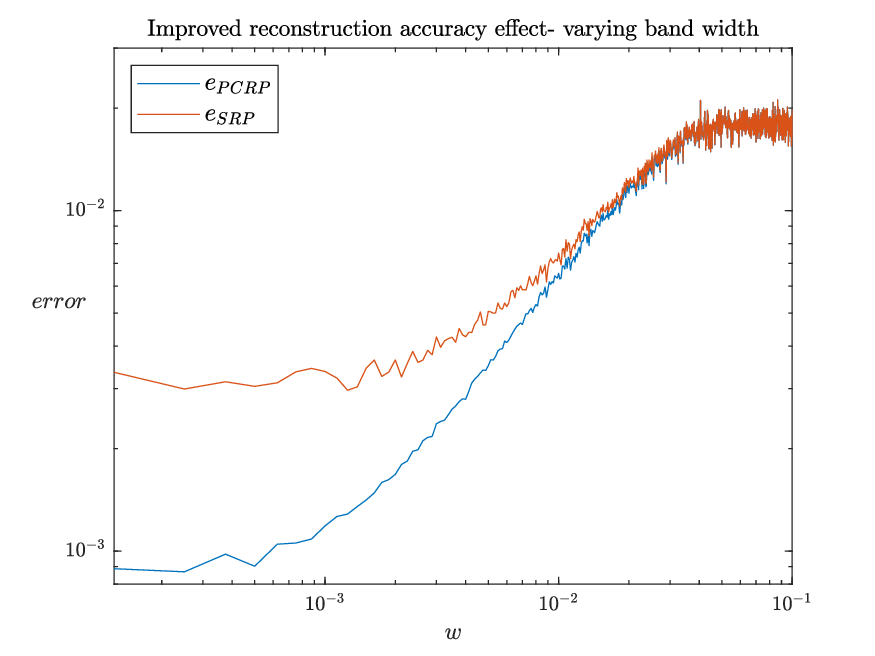}
 		\caption{}
		\label{fig.band.width}
\end{figure}

Figure \ref{fig.band.width} shows the results of our numerical experiments
which are arranged as follows: We fix the signal $F=\frac{1}{2}\delta(x+0.05)+\frac{1}{2}\delta(x-0.05)$
and a noise level of $\epsilon = h^3 = (0.1/2)^3$.
%
The feasible sets $\Omega^x_{w}$ in the node space are the strips transversal
to $S_{2d-1}(F)$,
$$\Omega^x_{w}=\{(x_1,x_2), \ |x_1+x_2|\le w\}$$
(as seen in Figure \ref{fig.improved1}, highlighted in pink).
For uniform  $\e$-noise
(i.e. the measured moments are uniformly distributed inside	the $\e$-cube in the moments space),
we plot the averages $e_{SRP},e_{PCRP}$
of the reconstruction error of the two procedures as a function of the width $w$ of $\Omega^x_{w}$.

As seen in the figure,
the advantage of the PCRP grows as the size $w$ decreases.
For values of $w \sim h^2$ or less the $PCRP$ attains a reconstruction error of $\sim h^2$ while
the $SRP$ attains a reconstruction error of $\sim h$.

\section{Prony mapping and its inversion}\label{Sec:Prony.mapping}

\subsection{Prony mapping}

\bd\label{Prony.Map}
The Prony mapping $PM=PM_{d}:{\cal P}_d\to {\cal M}_{d}$ is given by
$$
PM(F)= \mu = (\mu_0,\ldots,\mu_{2d-1}) \in {\cal M}, \ \mu_k=m_k(F), \ k=0,\ldots,2d-1.
$$
\ed
For $F\in {\cal P}$ the problem of its reconstruction from the exact moment measurements
$\mu = (\mu_0,\ldots,\allowbreak \mu_{2d-1}) \in {\cal M},$ is the problem of inverting the
Prony mapping $PM.$ In this paper we always assume that this inversion (when defined) is accurate.

\smallskip

Consider the noisy measurements $\mu'=(\mu'_0,\ldots,\mu'_{2d-1})\in {\cal M}$ of the moments of $F$.
By our assumption, the measurement error of each of the moments $m_k(F)$ does not exceed $\e$,
i.e. $|\mu'_k-\mu_k(F)|\le \e$. Equivalently, the noisy measurement $\mu'$ may fall at any point in the cube
\be\label{eq.epsilon.cube}
Q_{\e}(\mu)=\{\mu'=(\mu'_0, \ldots, \mu'_{2d-1}) \in {\cal M}, \ |\mu'_k - \mu_k|\le \e, \ k=0,1,\ldots,2d-1\}.
\ee
Consequently, the $\e$-error set $E_\e(F)$ is the preimage $PM^{-1}(Q_\e(\mu))\subset {\cal P}.$

\subsection{Inverse Function theorem and its consequences}\label{Sec:Inv.Fn.Thm}

Our first result describes the inversion of the Prony mapping in a neighborhood of a
``regular point'', i.e. of a signal $G$ with all its $d$ nodes well separated,
and with all its amplitudes bounded and well separated from zero.
This result is, essentially, a direct application of the ``quantitative inverse function theorem''
(see, for instance, \cite{hubbard2015vector}, page 264, Theorem 2.10.7 or \cite{friedland2015doubling}, Theorem 3.2)
combined with the estimates of the norm of the Jacobian of the Prony mapping and the norm of its inverse.

Assume that the nodes $x_1,\ldots,x_d$ of a signal $G$ all belong to the interval $I=[-1,1]$,
and for a certain $\eta$ with $0<\eta\leq \frac{2}{d-1}$, $d>1$,
the distance between the neighboring nodes $x_j,x_{j+1}, \ j=1,\ldots,d-1,$ is at least $\eta$.
We also assume that for certain $m,M$ with $0<m<M$, the amplitudes $a_1,\ldots,a_d$ satisfy $m\leq |a_j|\leq M, \ j=1,\ldots,d$.
We call such signals $(\eta,m,M)$-regular. We distinguish (as above) the parameter and the moment spaces of the model signals $G$,
denoting them by $\bar {\cal P},\bar{\cal M},$ respectively.
For $G\in \bar {\cal P}$ we denote by $\nu=(\nu_1,\ldots,\nu_{2d-1})$ its Prony image $PM(G)\in \bar{{\cal M}}.$

For a matrix $B = [b_{ij}]$, we denote by $\|B\|$ its maximum norm:
$$
	\|B\| = \max_i{\sum_{j}|b_{ij}|}.
$$

\bt\label{thm:inv.fn}
Let $G$ be an $(\eta,m,M)$-regular signal then there exist positive constants $R,C_1,\allowbreak C_2,C_3,C_4$ (given
explicitly below and in Appendix \ref{appendix.Quantitative}) depending only on $d,\eta,m,M$ such that

\begin{enumerate}[leftmargin=*]
	\item
		The Jacobian $J$ at $G$ of the Prony mapping $PM$ is invertible, with
		$$
			\|J^{-1}\|\leq C_1, \|J\|\le C_2.
		$$
	\item
		The inverse mapping $PM^{-1}$ exits on the cube $Q_R(\nu)$
		of size $R$ centered at $\nu\in \bar {\cal M},$ and provides a diffeomorphism of
		$Q_R(\nu)$ to $\Omega_R(G)= PM^{-1}(Q_R(\nu))$. For each $\nu',\nu'' \in Q_R(\nu)$
		$$
			C_3\|\nu''-\nu'\|\le \|PM^{-1}(\nu'')-PM^{-1}(\nu')\|\le C_4\|\nu''-\nu'\|.
		$$
\end{enumerate}
\et
\begin{proof}
Let $J=J(G)$ denote the Jacobian matrix of $PM$ at a (regular) signal $G$,
$$J_{k,j} =
\begin{cases}
	\frac{\partial m_k(G)}{\partial a_j}=x_j^k, & k=0,\ldots,2d-1, \tab j=1,\ldots,d, \\
	\frac{\partial m_k(G)}{\partial x_j}=ka_jx_j^{k-1}, & k=0,\ldots,2d-1, \tab j=d+1,\ldots,2d. \\  	
\end{cases}$$
The matrix $J$ admits the following factorization (about factorization of the Prony Jacobian see also
\cite{batenkov2013accuracy})
\begin{equation}\label{eq.model.jacobian.factorization}
			J=
            \begin{bmatrix}
                    1 & .. & 1 & 0 & .. & 0\\
                    x_1 & .. & x_d & 1 & .. & 1\\
                    . & .. & . & . & .. & .\\
                    x_1^{2d-1} & .. & x_d^{2d-1} & (2d-1)x_1^{2d-2} & .. & (2d-1)x_d^{2d-2}
            \end{bmatrix}
            \begin{bmatrix}
                    I_d & 0\\
                    0 & D
            \end{bmatrix},
\end{equation}
where $D=diag(a_1,\ldots,a_d)$ is a $d \times d$ diagonal matrix with the amplitudes on the diagonal and
$I_d$ is the $d\times d$ identity matrix. Denote the left hand matrix in this factorization by $U_{2d}$.
This is a special type of a confluent Vandermonde matrix, the norm of its inverse,
which is important in our estimates, was studied in \cite{gautschi1962inverses}.

\bt[Gautschi, \cite{gautschi1962inverses}, Theorem 3]
	$$
		\|U_{2d}^{-1}\| \le \max_{1 \le \lambda \le d} \; b_\lambda \left(\prod_{j=1,j\ne \lambda}^d \frac{1+|x_j|}{|x_{\lambda}-x_j|}\right)^2,
	$$
	$$
		b_{\lambda} = \max \left[1+|x_\lambda|,1+2(1+|x_{\lambda}|)\sum_{j=1,j\ne \lambda}^d \frac{1}{|x_\lambda - x_j|}\right].
	$$
\et

Based on the above, for $U_{2d}$ formed by the nodes of an $(\eta,m,M)$-regular signal, that is $|x_i| \le 1$ and $|x_i-x_j|\ge \eta$,
it is straight forward to bound $\|U_{2d}^{-1}\|$ in terms of the constants $\eta,d$.
The following proposition (given without proof) provide such upper bound.
\begin{proposition}\label{prop.uuper.bound.confluent.vandermonde}
	Let $|x_i| \le 1$ and $|x_i-x_j|\ge \eta$ then
	\begin{align*}
		\|U_{2d}^{-1}\| \le (1 +4\eta^{-1}(\ln(d)+1))\left(\frac{\eta^{-d+1}2^{d-1}}{\left(\lfloor
		\frac{d-1}{2}\rfloor!\right)^2} \right)^2 .
	\end{align*}
\end{proposition}
Now using proposition \ref{prop.uuper.bound.confluent.vandermonde} and the factorization equation
\eqref{eq.model.jacobian.factorization} we have that
\begin{equation}\label{eq.infinity.norm.bound}
	\|J^{-1}\| \le \max [1,m^{-1}](1 +4\eta^{-1}(\ln(d)+1))\left(\frac{\eta^{-d+1}2^{d-1}}{\left(\lfloor
	\frac{d-1}{2}\rfloor!\right)^2} \right)^2 = C_1(m,\eta, d)=C_1.
\end{equation}
In addition, for an $(\eta,m,M)$-regular signal, a direct calculation shows that
$$\|J\| \le d+ M(2d-1)d=C_2 .$$
This conclude the proof of statement 1 of Theorem \ref{thm:inv.fn}.

\bigskip

The second statement of Theorem \ref{thm:inv.fn} follows from
``quantitative inverse function theorem'' (see, \cite{hubbard2015vector}, Theorem
2.10.7 or \cite{friedland2015doubling}) taking into account that in this result the constants $C_3,C_4$ and
$R$ are given in terms of upper bounds on $\|J^{-1}\|, \|J\|$ and a local upper bound on the magnitude of the second
derivatives of $PM$. The latter can be easily obtained in terms of $d,\eta,m,M$.
The required constants $C_3,C_4$ and
$R$ are derived explicitly in appendix \ref{appendix.Quantitative}.
This completes the proof of Theorem \ref{thm:inv.fn}.
\end{proof}

\smallskip

Let us denote $\O_R(G)\subset \bar {\cal P}$ as the preimage $PM^{-1}(Q_R(\nu))$.
We give an equivalent formulation of Theorem \ref{thm:inv.fn}, in terms of the moment coordinates.

\bd\label{def:moment.coord.dist}
For $G$ a regular signal as above, and $G'$ denoting signals near $G$,
the moment coordinates are the functions $f_k(G')=m_k(G')-m_k(G),\; k=0,...,2d-1$.
The moment metric $d(G',G'')$ on $\bar {\cal P}$ is defined through the moment coordinates as
$$
d(G',G''):=\max_{k=0}^{2d-1}|m_k(G'')-m_k(G')| = \|PM(G')-PM(G'')\|.
$$
\ed

\bc\label{cor:coord.moments}
Let $G$ be a regular signal as above. Then the moment coordinates
form a regular analytic coordinate system on $\O_R(G)$.
The moment metric $d(G',G'')$ is bi-Lipschitz equivalent on $\O_R(G)$ to the maximum metric \;\;$\|G''-G'\|$
in $\bar{\cal{P}}$:
$$
C_3 d(G',G'')\le \|G''-G'\|\le C_4 d(G',G'').
$$
\ec
\begin{proof}
	It follows directly from Theorem \ref{thm:inv.fn}.
\end{proof}
\section{The geometry of the error set for nodes forming an $h$-cluster}\label{Sec:error.set.geometry}

We use regular signals $G$ as above, as a model for signals with a
``regular cluster'':
For $F\in {\cal P}$ with $h=h(F)$ and $\kappa=\kappa(F)$ (i.e. $F$ having its nodes cluster in an interval of size
$h$ and center $\kappa$), we say that $F$ forms an $(h,\kappa,\eta,m,M)$-regular cluster if $G=\Psi_{\kappa,h}(F)$ is an
$(\eta,m,M)$-regular signal. Explicitly $F$ forms an $(h,\kappa,\eta,m,M)$-regular cluster if its amplitudes
$a_1,\ldots,a_d$ satisfy $m\leq |a_j|\leq M, \ j=1,\ldots,d,$ and the distance between the neighboring nodes
$x_j,x_{j+1}, \ j=1,\ldots,d-1,$ is at least $\eta h$.
We formulate our main results in terms of the model signal $G$.

\smallskip

For any $G \in \bar{{\cal P}}$ and $\e,\alpha>0$ we define the following geometric objects:
\bd
	Define $\Pi_{\e,\alpha}(G) \subset \bar{{\cal P}}$ as the parallelepiped, in moments coordinates,
	consisting of all signals
	$G' \in \bar{\cal P}$ satisfying
	the inequalities $$ |m_k(G')-m_k(G)|\leq \e \alpha^{k}, \ k=0,\ldots,2d-1.
$$	
\ed
\bd\label{def.part.prony.ver}
	For each $0 \le q \le 2d-1$, define $S_{q,\e,\alpha}(G)$ as the part of the Prony variety $S_q(G),$
	consisting of all signals $G'\in S_q(G)$ with
	$$
		|m_k(G')-m_k(G)|\leq \e \alpha^{k}, \ k=q+1,\ldots,2d-1.
	$$
\ed


\subsection{The case of a zero shift}\label{Sec:zero.shidft}

Theorem \ref{thm.error.set.geometry.no.shift} below describes the set $\bar E_\e(F)\subset \bar {\cal P}$,
under an additional assumption that there is no shift.
In this case the description becomes especially transparent.
The effect of a non-zero shift $\kappa$ is described in Section \ref{Sec:non.zero.shidft} below.
In particular, a version of Theorem \ref{thm.error.set.geometry.no.shift} with a non-zero shift is given in Theorem \ref{thm.error.set.geometry.shift}.

\smallskip


\bt\label{thm.error.set.geometry.no.shift}
Let $F\in {\cal P}$ form an $(h,0,\eta,m,M)$-regular cluster and let $G=\Psi_{0,h}(F)$ be the model signal for
$F$. Then:
\begin{enumerate}[leftmargin=*]
  \item
  		For each positive $\e$ we have
	  	$$
			\bar{E}_\e(F)=\Pi_{\e,\frac{1}{h}}(G).
		$$
  \item For each positive $\e\le Rh^{2d-1}$,
		$\bar E_\e(F)$ is contained within the $\Delta_q$-neighborhood of the part of the Prony variety $S_{q,\e,\frac{1}{h}}(G)$, for
 		$$
 			\Delta_q=C_4\left(\frac{1}{h}\right)^{q}\e.
 		$$
 		The constants $R,C_4$ are defined in Theorem \ref{thm:inv.fn} above.
   	
\end{enumerate}
%
%
\et
\begin{remark}
	Assume that the measurement
	error $\e\le Rh^{2d-1}$.  By Corollary \ref{cor:coord.moments} we have
	that the metric induced by the moments is equivalent to maximum metric on $\bar{{\cal P}}$.
	Combing this with statement 1 of Theorem \ref{thm.error.set.geometry.no.shift},
	we obtain that the error set $\bar{E}_\e(F)$ is a ``deformed'' parallelepiped in standard coordinates of $\bar{{\cal P}}$.
	See figures \ref{fig.h01} and \ref{fig.h05} in subsection \ref{Sec:setting.sketch}.
\end{remark}

\begin{proof}[Proof of Theorem \ref{thm.error.set.geometry.no.shift}.]
Denote by $M\bar{E}_\e(F) = PM(\bar{E}_\e(F)) \subset \bar {\cal M}$
the set of all the possible errors in the moments $m_k(G), \ k=0,1,\ldots,2d-1,$
corresponding to the errors not exceeding $\e$ in the moments of $F$.

\smallskip

Consider the scaling transformation $SC_\alpha,$ which acts on signals $F$
via scaling of the {\it nodes} of $F$: $SC_\alpha(F)(x)=\frac{1}{\alpha}F(\frac{x}{\alpha}).$
For $F(x)= \sum_{j=1}^{d}a_{j}\delta\left(x-x_{j}\right)$
we have that $SC_\alpha(F)=\sum_{j=1}^{d}a_{j}\delta\left(x-\alpha x_{j}\right),$ and therefore
\begin{equation}\label{eq.scale}
m_k(SC_\alpha(F))=\sum_{j=1}^{d}a_{j}\alpha^k x^k_{j}=\alpha^k m_k(F).
\end{equation}
Accordingly, we define the scaling transformation $SC^*_\alpha: {\cal M}\to {\cal M}$ on the moment space as follows: for $\mu=(\mu_0,\ldots,\mu_{2d-1})$
\be\label{eq.scalling.def}
SC^*_\alpha(\mu)=\nu = (\nu_0,\ldots,\nu_{2d-1}), \ \nu_k=\alpha^k \mu_k, \ k=0,\ldots,2d-1.
\ee
With these definitions we have for all $F' \in {\cal P}$
\be\label{eq.scalling}
PM(SC_\alpha(F'))= SC^*_\alpha(PM(F')).
\ee

\smallskip

 For the model signal $G$ we have $G=SC_{\frac{1}{h}}(F)$. Set $\mu =PM(F)$. Accordingly, the set $M\bar{E}_\e(F)$ of the possible measurements
 for the moments of $G$ is $SC^*_{\frac{1}{h}}(Q_\e(\mu)).$ The initial moment error set $E_\e(F)$ is the $\e$-cube
 $Q_\e(\mu)$,
$$
	Q_{\e}(\mu)=\{\mu'=(\mu'_0, \ldots, \mu'_{2d-1}) \in {\cal M}, \ |\mu'_k - \mu_k|\le \e, \ k=0,1,\ldots,2d-1\}.
$$
Consequently, $M\bar{E}_\e(F)$ is a coordinate parallelepiped
\begin{equation}\label{eq.no.shift.ME}
	M\Pi_{\e,\frac{1}{h}}(\nu):=\{\nu' \in {\cal M}, |\nu'_k-\nu_k|\le \e \left(\frac{1}{h}\right)^{k}, \ k=0,1,\ldots,2d-1\}.	
\end{equation}
The error set $\bar E_\e(F)\subset \bar {\cal P}$ is the preimage
$$
\bar E_\e(F)=PM^{-1}(M\bar{E}_\e(F))=PM^{-1}(M\Pi_{\e,\frac{1}{h}}(\nu)) = \Pi_{\e,\frac{1}{h}}(G).
$$
This concludes the proof of the first part of Theorem \ref{thm.error.set.geometry.no.shift}.

\medskip

We now prove the second part of Theorem \ref{thm.error.set.geometry.no.shift}. By part one of the theorem we already know
that $M\bar{E}_{\e}(F)=M\Pi_{\e,\frac{1}{h}}(\nu)$ is the parallelepiped given in \eqref{eq.no.shift.ME}.
On the other hand $PM(S_{q,\e,\frac{1}{h}}(G))$
is the projection of $M\Pi_{\e,\frac{1}{h}}(\nu)$ into
the last $2d-q-1$ coordinates (in the moments coordinate system centered at $G$). Hence
$$\max_{\nu' \in M\bar{E}_{\e}(F)} \min_{\nu'' \in PM(S_{q,\e,\frac{1}{h}}(G))  } \|\nu'-\nu''\| =
\left(\frac{1}{h}\right)^{q}\e .$$ In order to apply Theorem \ref{thm:inv.fn} and Corollary \ref{cor:coord.moments}
we have to check that the parallelepiped $M\bar{E}_{\e}(F)=M\Pi_{\e,\frac{1}{h}}(\nu)$ is contained
in the cube $Q_R(\nu)$ of size $R$ centered at $\nu\in \bar {\cal M}.$ The maximal edge of $M\Pi_{\e,\frac{1}{h}}(\nu)$
has length $\e h^{-2d+1}$, and hence for $\e \le Rh^{2d-1}$ the required inclusion holds. Now, by applying Corollary \ref{cor:coord.moments},
we get
$$\max_{G' \in \bar{E}_{\e}(F)} \min_{G'' \in S_{q,\e,\frac{1}{h}}(G)  } \|G'-G''\| = C_4
\left(\frac{1}{h}\right)^q\e.$$
\end{proof}

\subsection{The case of a non-zero shift}\label{Sec:non.zero.shidft}
For a signal $G \in \bar{{\cal P}}$ recall that the parallelepiped $\Pi_{\e,\alpha}(G) \subset \bar{{\cal P}}$,
is the set of all signals $G' \in \bar{\cal P}$ satisfying
$$ |m_k(G')-m_k(G)|\leq \e \alpha^{k}, \ k=0,\ldots,2d-1.$$

\bt\label{thm.error.set.geometry.shift}
Let $F\in {\cal P}$ form an $(h,\kappa,\eta,m,M)$-regular cluster and let $G=\Psi_{\kappa,h}(F)$ be the model signal for
$F$. Set $\e'=(1+|\kappa|)^{-2d+1} \e$ and $h'=\frac{h}{1+|\kappa|}$. Then:
\begin{enumerate}[leftmargin=*]
  \item For any $\e>0$, the error set $\bar{E}_\e(F)$ is
  contained between the following two parallelepipeds in the moment coordinates:
  		$$
			 \Pi_{ \e',\frac{1}{h}}(G) \subset \bar E_\e(F) \subset \Pi_{\e,\frac{1}{h'}}(G),	
		$$
		where
		\begin{align*}
			\Pi_{\e',\frac{1}{h}}(G) &= \left\{G' \in \bar{{\cal P}} : |m_k(G)-m_k(G')|\le
			(1+|\kappa|)^{-2d+1} \e \left(\frac{1}{h}\right)^k,\;k=0,\ldots,2d-1 \right\},\\
			\Pi_{\e,\frac{1}{h'}}(G) &= \left\{G' \in \bar{{\cal P}} : |m_k(G)-m_k(G')|\le  \e
			\left(\frac{1+|\kappa|}{h}\right)^k,\;k=0,\ldots,2d-1\right\}.
		\end{align*}
  \item For any $\e \le R h'^{2d-1}$, the error set $\bar{E}_{\e}(F)$
   		is contained within the $\Delta'_q$-neighborhood of the part of the Prony variety $S_{q,\e,\frac{1}{h'}}(G)$, for
   		$$
 			\Delta'_q=C_4\left(\frac{1}{h'}\right)^{q}\e.
 		$$
 		The constants $R,C_4$ are defined in Theorem \ref{thm:inv.fn} above.
\end{enumerate}

\et

\begin{proof}[Proof Theorem \ref{thm.error.set.geometry.shift}:]
	Let us describe the effect of a shift transformation in ${\cal P}$ and in ${\cal M}$. Define the shift
transformation $SH_\kappa: {\cal P}\to {\cal P}$ of the parameter space by $SH_\kappa(F)(x)=F(x+\kappa)$.
The following proposition describes the action of the coordinate shift on the moments
of general spike-trains:

\bp\label{prop:shift3}
$$
m_k(F)=\sum_{l=0}^k \binom{k}{l}(\kappa)^{k-l} m_l(SH_\kappa(F)), \ m_k(SH_\kappa(F)) = \sum_{l=0}^k \binom{k}{l}(-\kappa)^{k-l} m_l(F).
$$
\ep
\begin{proof}
For $F(x)=\sum_{j=1}^d a_j\delta (x-x_j)\in {\cal P}$ we get
$$
m_k(SH_\kappa(F))= \sum_{j=1}^d a_j (x_j-\kappa)^k =\sum_{j=1}^d a_j \sum_{l=0}^k \binom{k}{l}(-\kappa)^{k-l}x_j^l =
$$
$$
= \sum_{l=0}^k \binom{k}{l}(-\kappa)^{k-l} \sum_{j=1}^d a_j x_j^l=      \sum_{l=0}^k \binom{k}{l}(-\kappa)^{k-l} m_l(F).
$$
Replacing $\kappa$ by $-\kappa$ we get the second expression.
\end{proof}

Accordingly, we define the shift transformation $SH^*_\kappa: {\cal M}\to {\cal M}$ as the following linear transformation
on the moment space: For $\mu'=(\mu'_0,\ldots,\mu'_{2d-1})\in {\cal M}$
$$
SH^*_\kappa(\mu')=\nu = (\nu_0,\ldots,\nu_{2d-1}), \ \nu_k = \sum_{l=0}^k \binom{k}{l}(-\kappa)^{k-l} \mu_l, \ k=0,1,\ldots,2d-1.
$$

\smallskip

Proposition \ref{prop:shift3} shows that the shift transformations $SH_\kappa$ and $SH^*_\kappa$, and the Prony mapping
$PM$ satisfy the following identity:
\be\label{eq:shift3}
PM(SH_\kappa(F))=SH^*_\kappa(PM(F)).
\ee

\smallskip


Since $SH^*_\kappa$ is a linear transformation we will omit the parentheses and write $SH^*_\kappa\mu$ instead of $SH^*_\kappa(\mu)$.
We extend this rule to every linear transformation $T$ and write $Tv$ instead of $T(v)$.
We have the following bounds for the norms of $SH^*_\kappa$ and $SH^{*-1}_\kappa$:

\bp\label{prop:shift.norm3}
The shift transformation $SH^*_\kappa: {\cal M}\to {\cal M}$ satisfies for each $0 \le k\le 2d-1$
$$\max_{\mu\in {\cal M},\|\mu\|=1} |(SH_{\kappa}^{*}\mu)_k| \le (1+|\kappa|)^k,\;\;
\max_{\mu\in {\cal M},\|\mu\|=1} |(SH_{\kappa}^{*-1}\mu)_k| \le (1+|\kappa|)^k,$$
where $(SH^{*}_\kappa\mu)_k,\;(SH^{*-1}_\kappa\mu)_k$, denotes the $k^{th}$ coordinate of $SH^{*}\mu$ and
$SH^{*-1}\allowbreak \mu$ respectively.
As a result 	
$$
	\|SH^*_\kappa\|,\; \|SH^{*-1}_\kappa\| \le (1+|\kappa|)^{2d-1}.
$$
\ep
\pr
For $\mu=(\mu_0,\ldots,\mu_{2d-1})\in {\cal M},$ with $\|\mu\|=1$, we have for each $0 \le k \le 2d-1$
\begin{align*}
	\|(SH^*_\kappa\mu)_{k}\| \le \max_{k=0,\dots,2d-1} \ \sum_{l=0}^k \binom{k}{l}|\kappa|^{k-l} = (1+|\kappa|)^k
\end{align*}

The inequality for $|(SH_{\kappa}^{*-1}\mu)_k|$ follows by noting that $SH^{*-1}_{\kappa}=SH^*_{-\kappa}$.
$\square$

\smallskip

Let $F\in {\cal P}$, as above, form an $(h,\kappa,\eta,m,M)$-regular cluster
and put $PM(F)=\mu$. Then, by identities \eqref{eq:shift3} and \eqref{eq.scalling}
$$PM(\bar{E}_\e(F))=PM(SC_{\frac{1}{h}}SH_{\kappa}E_{\e}(F))=SC^{*}_{\frac{1}{h}}SH^*_\kappa
PM(E_{\e}(F))=SC^{*}_{\frac{1}{h}}SH^*_\kappa Q_\e(\mu).$$

Put $\xi=(\xi_0,\ldots,\xi_{2d-1})=SH^{*}_{\kappa}(F)$. By Proposition \ref{prop:shift.norm3} we have
$$SH^*_\kappa Q_\e(\mu) \subset M\Pi_{\e,1+|\kappa|} = \{\xi' \in {\cal M},\; |\xi'_k-\xi_k|\le \e (1+|\kappa|)^{k}, \ k=0,1,\ldots,2d-1\}.$$
Put $\nu=(\nu_0,\ldots,\nu_{2d-1})=PM(G)$. Then again by identities \eqref{eq:shift3} and \eqref{eq.scalling}
$$\nu=PM(G)=PM(\Psi_{\kappa,h}(F))=PM(SC_{\frac{1}{h}}SH_{\kappa}F)=$$
$$=SC^{*}_{\frac{1}{h}}SH^{*}_{\kappa}PM(F)=SC^{*}_{\frac{1}{h}}SH^{*}_{\kappa}\mu.$$
Using the above and by definition of $SC^{*}$ we get
\begin{align*}
	SC^{*}_{\frac{1}{h}}SH_{\kappa}^* Q_\e(\mu) &\subset SC^{*}_{\frac{1}{h}}M\Pi_{\e,1+|\kappa|}\\
												&= \{\nu' \in {\cal M},\; |\nu'_k-\nu_k|\le \e \left(\frac{1+|\kappa|}{h}\right)^{k}, \ k=0,1,\ldots,2d-1\}.
\end{align*}
This proves that $\bar{E}_\e(F) \subset \Pi_{\e,\frac{1+|\kappa|}{h}}(G)=\Pi_{\e,\frac{1}{h'}}(G)$.

\smallskip

We now prove that for $\e'=(1+|\kappa|)^{2d-1}\e$, $\Pi_{\e',\frac{1}{h}}(G)
\subset \bar{E}_\e(F)$.
By Proposition \ref{prop:shift.norm3} the norm  of the inverse shift transformation has the following lower bound,
$\|SH^{*-1}_\kappa\|\le (1+|\kappa|)^{2d-1}$. Then $SH_{\kappa}^* Q_\e(\mu) \supset
\left(\frac{1}{1+|\kappa|}\right)^{2d-1} Q_\e(\xi)$. Applying the scaling transformation $SC^*_{\frac{1}{h}}$ we get
$$
	SC^*_{\frac{1}{h}} SH_{\kappa}^* Q_\e(\mu) \supset \left(\frac{1}{1+|\kappa|}\right)^{2d-1} SC^*_{\frac{1}{h}}Q_\e(\xi)
			=\left(\frac{1}{1+|\kappa|}\right)^{2d-1}  M\Pi_{\e,h}(\nu).$$
Therefore $\bar{E}_\e(F) \supset
\Pi_{(1+|\kappa|)^{-2d+1}\e,\frac{1}{h}}(G)=\Pi_{\e',\frac{1}{h}}(G)$.
This completes the proof of the first statement of Theorem \ref{thm.error.set.geometry.shift}.
%

\smallskip

Next we prove the second statement of Theorem \ref{thm.error.set.geometry.shift}.
For a given $0 \le q \le 2d-1$, we need to show that the error set $\bar{E}_\e(F)$
is contained in an order of $h^{-q}\e$ neighborhood of the part of the Prony variety $S_{q,\e,\frac{1}{h'}}(G)$, i.e.
$$
	\max_{G' \in \bar{E}_\e(F)}\;\;\min_{G'' \in S_{q,\e,\frac{1}{h'}}(G)} \|G'-G''\|\le
	C_4\left(\frac{1}{h'}\right)^{q}\e .
$$

Set as above  $M\bar{E}_\e(F) = PM(\bar{E}_\e(F)) \subset  {\cal M}$ and $\nu=(\nu_0,\ldots,\nu_{2d-1})=PM(G)$.
By statement 1 of Theorem \ref{thm.error.set.geometry.shift},
$$M\bar{E}_{\e}(F) \subset M\Pi_{\e,\frac{1}{h'}}(\nu) = \{\nu' \in {\cal M},\; |\nu'_k-\nu_k|\le \e
\left(\frac{1+|\kappa|}{h}\right)^{k}, \ k=0,1,\ldots,2d-1\}.$$ On the other hand $PM(S_{q,\e,\frac{1}{h'}}(G))$
is the projection of $M\Pi_{\e,\frac{1}{h'}}(\nu)$ into
the last $2d-q-1$ coordinates (in the moments coordinate system centered at $G$). Hence
$$\max_{\nu' \in M\bar{E}_{\e}(F)} \;\; \min_{\nu'' \in PM(S_{q,\e,\frac{1}{h'}}(G))  } \|\nu'-\nu''\|
\le \max_{\nu' \in M\Pi_{\e,\frac{1}{h'}}(\nu)}\;\; \min_{\nu'' \in PM(S_{q,\e,\frac{1}{h'}}(G))  } \|\nu'-\nu''\|=
\left(\frac{1}{h'}\right)^{q}\e .$$

We now want to apply equivalence of the moments metric on $\cal \bar{M}$ and the maximum metric on $\cal
\bar{P}$ given in Corollary \ref{cor:coord.moments}.
For this purpose we need to check that $M\bar{E}_\e(F) \subset Q_R(\nu)$. Again by statement 1 of the theorem
$M\bar{E}_\e(F) \subset M\Pi_{\e,\frac{1}{h'}}(\nu)$.
By assumption we have that $\e \le R h'^{2d-1}$ then
$$M\bar{E}_{\e}(F) \subset M\Pi_{\e,\frac{1}{h'}}(\nu) \subset Q_R(\nu).$$
Now applying Corollary \ref{cor:coord.moments} we get
$$\max_{G' \in \bar{E}_{\e}(F)} \;\;\min_{G'' \in S_{q,\e,\frac{1}{h'}}(G)} \|G'-G''\| \le C_4
\left(\frac{1+|\kappa|}{h}\right)^q.$$ This concludes the proof of statement 2 of Theorem
\ref{thm.error.set.geometry.shift}.
\end{proof}

\section{Worst case reconstruction error}\label{sec:worst.case.reconstruction.error}

We now consider the worst case reconstruction error of a signal $F=(a,x)$ forming an $(h,\kappa,\allowbreak \eta,m,M)$-regular cluster. Define the worst
case reconstruction error of $F$ as

$$
\rho(F,\e)= \max_{F'\in E_\e(F)}\|F'-F\|.
$$
In a similar way we define $\rho_a(F,\e)$ and $\rho_x(F,\e)$ as the worst case errors in reconstruction of the amplitudes and nodes of $F$ respectively:
$$
\rho_a(F,\e)= \max_{F'=(a',x')\in E_\e(F)}\|a'-a\|, \ \rho_x(F,\e)= \max_{F'=(a',x')\in E_\e(F)}\|x'-x\|.
$$
We show that for $\e\le O(h^{-2d+1})$, $\rho(F,\e),\;\rho_a(F,\e)$ are of order $h^{-2d+1}\e$ and
$\rho_x(F,\e)$ is of order $h^{-2d+2}\e$.

\smallskip

The following theorem provide tight, up to constants, upper bounds on $\rho(F,\e),\allowbreak
\;\rho_x(F,\e),\allowbreak\;\rho_a(F,\e)$.
It is a direct consequence of the geometry of the error set presented in Theorem \ref{thm.error.set.geometry.shift}.
\bt[Reconstruction error upper bound]\label{thm.upper.bound}
Let $F\in {\cal P}_d$ form an $(h,\kappa,\eta,m,M)$-regular cluster.
Then for each positive $\e\le \left(\frac{h}{1+|\kappa|}\right)^{2d-1}  R$ the following bounds for the worst case reconstruction errors are
valid:
$$
	 \rho(F,\e),\; \rho_a(F,\e)  \le C_4 \left(\frac{1+|\kappa|}{h}\right)^{2d-1} \e,\tab \rho_x(F,\e)  \le C_4 \left(\frac{1+|\kappa|}{h}\right)^{2d-1}h \e,
$$
where $C_4, R$ are the constants defined in Theorem \ref{thm:inv.fn}.
\et
\begin{proof}
	For $F=(a,x)$ as in the theorem, let $G=(a,\bar{x})=\Psi_{\kappa,h}(F)$ be the model signal of $F$. We define the model worst case reconstruction errors
	$\bar \rho(F,\e),\; \bar{\rho}_a(F,\e)$ and $\bar{\rho}_x(F,\e)$ by
	\begin{align*}
		\bar \rho(F,\e) &= \max_{G'\in \bar E_\e(F)}\|G'-G\|,\\
		\bar{\rho}_a(F,\e)&= \max_{G'=(a',x')\in \bar{E}_\e(F)}\|a'-a\|,\\
		\bar{\rho}_x(F,\e)&= \max_{G'=(a',x')\in \bar{E}_\e(F)}\|x'-\bar{x}\|.\\
	\end{align*}
	We define the model worst case reconstruction error $\tilde \rho(F,\e)$ in the moment metric by
	$$
	\tilde \rho(F,\e)= \max_{G'\in \bar E_\e(F)}d(G',G).
	$$
	By Theorem \ref{thm.error.set.geometry.shift}, the error set $\bar E_\e(F) \subset \Pi_{\e,\frac{1+|\kappa|}{h}}(G)$.	
	Therefore we have
	\begin{align}\label{eq.bound.equivalence}
		\tilde \rho(G,\e) &\le \max_{G'\in \Pi_{\e,\frac{1+|\kappa|}{h}}(G)}d(G',G)=\left(\frac{1+|\kappa|}{h}\right)^{2d-1} \e .
	\end{align}
	For $\e\le \left(\frac{h}{1+|\kappa|}\right)^{2d-1}R$ we have that
	$$PM(\bar{E}_{\e}(F)) \subset PM( \Pi_{\e,\frac{1+|\kappa|}{h}}(G)) \subset Q_R(PM(G)).$$
	We can therefore apply the equivalence of the moment and the maximum metrics given in Corollary
	\ref{cor:coord.moments} and get that
	\be\label{eq.upper.bound.model}
	 	\bar{\rho}(F,\e) \le C_4 \tilde \rho(F,\e) = C_4 \left(\frac{1+|\kappa|}{h}\right)^{2d-1} \e.
	\ee
	Since $\rho_a(F,\e), \; \rho_x(F,\e)$ are each the maximum of the projected errors into the
	amplitudes and nodes subspaces respectively, inequality \ref{eq.upper.bound.model} also implies that
	\begin{align}\label{eq.upper.bound.model.projected}
		 \bar{\rho}_a(F,\e),\;\bar{\rho}_x(F,\e) \le C_4 \left(\frac{1+|\kappa|}{h}\right)^{2d-1} \e.
	\end{align}
	
	\smallskip
	
	Now we return from $G$ to the original signal $F$, and from the model space $\bar {\cal P}$ to $\cal P$.
	In this transformation the amplitudes remain unchanged, while the nodes are multiplied by $h$ (and
	shifted by $\kappa$).
	Therefore inequalities \eqref{eq.upper.bound.model} and \eqref{eq.upper.bound.model.projected} implies that
	\begin{align*}\label{eq:upper.bound}
		\rho(F,\e),\;\rho_a(F,\e) \le C_4 \left(\frac{1+|\kappa|}{h}\right)^{2d-1}\e,\tab
		\rho_x(F,\e) \le C_4 \left(\frac{1+|\kappa|}{h}\right)^{2d-1} h \e.
	\end{align*}
\end{proof}

We now give lower bounds on the worst case reconstruction errors:
$\rho(F,\e),\;\allowbreak \rho^a(F,\e)$ and $\rho^x(F,\e)$ of the same order of the upper bounds given in Theorem
\ref{thm.upper.bound} above.

\bt[Reconstruction error lower bound]\label{thm.lower.bound}
Let $F\in {\cal P}_d$ form an $(h,\kappa,\eta,m,M)$-regular cluster, then there exist positive constants $K_3, K_4, C_6,
C_7$, depending only on $d,\eta,m,M$, such that:
\begin{enumerate}
  \item For each positive $\e\le C_6 h^{2d-1}$ we have the following lower bound on the worst case reconstruction error of the nodes of $F$
  		$$K_3 \left(\frac {1}{(1+|\kappa|)h}\right)^{2d-1} h \e \le \rho_x(F,\e).$$
  \item For each positive $\e\le C_7 h^{2d-1}$ we have the following lower bounds on the worst case reconstruction error of $F$ and the amplitudes of $F$
     $$K_4 \left(\frac{1}{(1+|\kappa|)h}\right)^{2d-1} \e \le \rho(F,\e),\;\rho_a(F,\e).$$
\end{enumerate}
\medskip
%
\et
\begin{proof}

Let $G=(a,\bar{x})=\Psi_{\kappa,h}(F)$ be the model signal of $F=(a,x)$.
Let $PM(G)=\nu=(\nu_0,\ldots,\nu_{2d-1})$. Consider now the Prony curve $S_{2d-2}(G)$ which is defined by the equations
$$m_k(G')=m_k(G)=\nu_k, \ k=0,\ldots,2d-2.$$
Assume that $\e \le Rh^{2d-1}$ and let $\e'=(1+|\kappa|)^{-2d+1} \e$.
By the choice of $\e$ we have
$$
  	PM(\Pi_{ \e',\frac{1}{h}}(G)) \subset Q_R(\nu).
$$

Then by Corollary \ref{cor:coord.moments} the moment coordinates form a regular analytic coordinate system on $\Pi_{
\e',\frac{1}{h}}(G)$.
We can therefore fix the signal $G_{LB}\subset \bar{\cal{P}}$ with moment coordinates $\nu_{LB}=(\nu_0,\ldots,\nu_{2d-2},\allowbreak \nu_{2d-1}+\e'	h^{-2d+1})$.
The signal $G_{LB}$ is one of the intersection points of the Prony curve $S_{2d-2}(G)$ and the boundary of the parallelepiped  $\Pi_{ \e',\frac{1}{h}}(G)$.

By Theorem \ref{thm.error.set.geometry.shift} we have that the error set
$$
	\Pi_{ \e',\frac{1}{h}}(G) \subset \bar E_\e(F),
$$
hence $G_{LB} \in \barE$. Once again by Corollary \ref{cor:coord.moments} the moment metric and the maximum metric on
$\barP$ are equivalent and we have
\be\label{lower.bound.G}
\|G_{LB}-G\| \ge C_3 \cdot d(G_{LB},G) = C_3 h^{-2d+1} \e'.
\ee

\smallskip

The rest of the proof is essentially devoted to the fact that the projection of the error into both the amplitudes and nodes is non degenerate
and to deriving specific constants that bound from below the size of these projections.

\smallskip

Let $G_{LB}=(\tilde{a},\tilde{x})$ with $\tilde{a}=(\tilde{a_1},\ldots,\tilde{a_d})$ and  $\tilde{x}=(\tilde{x_1},\ldots,\tilde{x_d})$.
We now prove that for this specific signal (and for $\e$ small enough), the errors in the amplitudes and in the
nodes, $\|\tilde{a}-a\|$ and $\|\tilde{x}-\bar{x}\|$, are bounded from below as required.

\smallskip

We study in more detail the structure of the Jacobian matrix of the Prony mapping
at (the regular signal) $G$.
\medskip

The Jacobian of $PM$ at the point $G=(a,\bar x)$ is given by the matrix $J=J(G)$:

\medskip
\begin{equation}\label{eq.model.jacobian}
			J=
            \begin{bmatrix}
                    1 & .. & 1 & 0 & .. & 0\\
                    \bar{x}_1 & .. & \bar{x}_d & a_1 & .. & a_d\\
                    . & .. & . & . & .. & .\\
                    \bar{x}_1^{2d-1} & .. & \bar{x}_d^{2d-1} & a_1(2d-1)\bar{x}_1^{2d-2} & .. &
                    a_d(2d-1)\bar{x}_d^{2d-2} \end{bmatrix},
\end{equation}
or $J=[J_{k,j}]$ with
$$J_{k,j} =
\begin{cases}
	\frac{\partial m_k(G)}{\partial a_j}=\bar{x}_j^k, & k=0,\ldots,2d-1, \tab j=1,\ldots,d, \\
	\frac{\partial m_k(G)}{\partial x_j}=ka_j\bar{x}_j^{k-1}, & k=0,\ldots,2d-1, \tab j=d+1,\ldots,2d. \\  	
\end{cases}$$
\smallskip

We use the following notation to refer to submatrix blocks of $J$. For $J$ as above, we index the rows of $J$ (corresponding to the moment functions
$m_0,\ldots,m_{2d-1}$) by $0,\ldots,2d-1$  and the columns of $J$ by $1,\ldots,2d$. We will denote by $J(m:n,i:j)$, $0 \le m \le n \le 2d-1$,
$1 \le i \le j \le 2d$, the block of $J$ formed by the intersection of the rows $m,\ldots,n$ and the columns $i,\ldots,j$ of $J$.

\smallskip

We now prove a lower bound for the worst case errors of the nodes of $G$.
\begin{proposition}\label{prop.lower.bound.nodes}
	For $G_{LB}=(\tilde{a},\tilde{x})$, $G=(a,\bar{x})$ as above, and for $\e\le C_6 h^{2d-1}$, it holds that
	$$\|\tilde{x}-\bar{x}\| \ge K_3 \left(\frac {1}{(1+|\kappa|)h}\right)^{2d-1} \e ,$$
	where $K_3, C_6$ are constants depending only on $d,\eta ,m , M,$ defined within the
	proof.
\end{proposition}

\begin{proof}
	Consider the upper left $d\times d$ block of $J$, $J_1=J(0:d-1,1:d)$ and the upper right block of $J$,
	$J_2=J(0:d-1,d+1:2d)$.

	\smallskip
	
	We will need the following preliminaries:
	
	\smallskip
	
	The next proposition bounds the remainder of the linear estimate of $PM$ near a regular signal $G$.
	\begin{proposition}\label{prop.jacobian.PM.diff}
			Let $G$ be an $(\eta,m,M)$-regular signal. Let $r\le \frac{1}{2d-1}$ and $G'$ be a signal such that
			$\|G'-G\|\le r$. Let $J=J(G)$ be the Jacobian
			matrix at $G$. Then
			$$\left\|\big(PM(G')-PM(G)\big) -J\cdot \big(G'-G \big)\right\| \le C_5(d,M)\cdot r \cdot\|G'-G\|,$$
			where $C_5=6(M+1)(2d-1)^2 d$.
	\end{proposition}
	The proof of Proposition \ref{prop.jacobian.PM.diff} is given as an intermediate step
	in the proof of the quantitative inverse function theorem version, see Appendix \ref{appendix.Quantitative}, Proposition \ref{prop.jacobian.PM.diff.G'.G}.
	\begin{proposition}\label{prop.matrix.diff}
		Let $A$ be a non-singular $d\times d$ matrix and $B$ be any non-zero $d\times d$ matrix.
		Let $v \ne 0,u \in \mathbb{R}^d$ such that  $\|Av+Bu\| \le \alpha \|v\|$	
		where $\|\cdot\|$ is any norm on $\mathbb{R}^d$. Then
		$$
			\|u\| \ge \|v\|\frac{1-\alpha\|A^{-1}\|}{\|A^{-1}\|\;\|B\|},
		$$
		where $\|A^{-1}\|,\; \|B\|$ are the induced matrix norms.
	\end{proposition}
	\begin{proof}
	Put $\omega = Av+Bu$, then
		$$
			\|v\| = \|A^{-1}(-Bu +\omega)\| \le \|A^{-1}\| (\|B\|\;\|u\|+\|w\|) \le
			\|A^{-1}\|(\|B\|\;\|u\|+\alpha \|v\|).
		$$
		Rearranging the above we get
		$
			\|v\|\frac{1-\alpha\|A^{-1}\|}{\|A^{-1}\|\;\|B\|} \le \|u\|.
		$
	\end{proof}
		
		Now let $P_{0,d}:
		\mathbb{R}^{2d} \rightarrow \mathbb{R}^{d}$ be the projection to the first $d$ coordinates, i.e.
		for $x=(x_0,\ldots,x_{2d-1}) \in \mathbb{R}^{2d}$, $P_{0,d-1}x =P_{0,d-1}(x) = (x_0,\ldots,x_{d-1})$.
		
		By Proposition
		\ref{prop.jacobian.PM.diff} we get that
		\begin{align}\label{eq.jacobian.pm.diff.nodes.1}
			\begin{split}
					C_5\|G_{LB}-G\|^2 &\ge \left\|\big(PM(G_{LB})-PM(G)\big) -J \cdot \big(G_{LB}-G \big) \right\| \\
					&\ge \left\|P_{0,d-1}\bigg(\big(PM(G_{LB})-PM(G)\big) -J \cdot \big(G_{LB}-G \big)\bigg)\right\|\\
					 &= \left\|J_1(\tilde{a}-a )+J_2(\tilde{x}- \bar{x})\right\|.
			\end{split}
		\end{align}
		
		
		We note that $J_1$ is a Vandermonde matrix with nodes $\bar{x}_1,\ldots,\bar{x}_d$.
		The following theorem bounds the norm of an inverse Vandermonde matrix.
		\bt[Gautschi, \cite{gautschi1962inverses}, Theorem 1]\label{thm.GAUTSCHI.1}
			Let $V_d = V_d(x_1,\ldots,x_d)$ be a $d \times d$ Vandermonde matrix, $V_{i,j}=x_j^i$, $i=0,\ldots,\;d-1$,
			$j=1,\ldots,\;d$, with distinct nodes. Then
			$$
				\|V_{d}^{-1}\| \le \max_{1 \le \lambda \le d} \; \prod_{j=1,j\ne \lambda}^d
				\frac{1+|x_j|}{|x_{\lambda}-x_j|}.
			$$
		\et
		\smallskip
		
		
		The nodes of $G$ satisfies $|\bar{x}_i| \le 1$ and for $i\ne j$,
		$|\bar{x}_i-\bar{x}_j|\ge \eta$. Based on Theorem \ref{thm.GAUTSCHI.1} we can bound
		the norm of $\|J_1^{-1}\|$ by a constant depending on the minimal separation of the nodes $\eta$
		and $d$. The next proposition, given without proof, is a direct consequence of Theorem \ref{thm.GAUTSCHI.1} above.
		\begin{proposition}\label{prop.uuper.bound.vandermonde}
			Let $V_d = V_d(x_1,\ldots,x_d)$ be a $d \times d$ Vandermonde matrix, $V_{i,j}=x_j^i$, $i=0,\ldots,\;d-1$ with
			$|x_i| \le 1$ and $|x_i-x_j|\ge \eta$ for each $1 \le i<j\le d$ then
			\begin{align*}
				\|V_{d}^{-1}\| \le \frac{\eta^{-d+1}2^{d-1}}{\left(\lfloor
				\frac{d-1}{2}\rfloor!\right)^2}  .
			\end{align*}
		\end{proposition}
		
		Therefore we can fix a constant $C_8 = C_8(\eta,d)$ such that
		$$\|J_1^{-1}\| \le C_8(\eta,d) \le \frac{\eta^{-d+1}2^{d-1}}{\left(\lfloor
				\frac{d-1}{2}\rfloor!\right)^2}.$$
		By a direct calculation we also have that
		$$\|J_2\|\le d(d-1)M .$$
	
		\smallskip
		
		By equation \eqref{lower.bound.G} $\|G_{LB}-G\| \ge C_3 h^{-2d+1} \e' $,
		$\e'=(1+|\kappa|)^{-2d+1} \e$.
		Hence, either $\|\tilde{x}-\bar{x}\| = \|G_{LB}-G\| \ge  C_3 h^{-2d+1} \e'$ and
		in this case setting $K_3 = C_3$ and $C_6=R$ we are done.
		Else,
		\be\label{eq.latter.case}
			\|\tilde{a}-a\| = \|G_{LB}-G\|.
		\ee
		
		\smallskip
			
		We continue under the assumption of equation $\eqref{eq.latter.case}$.
		From $\eqref{eq.latter.case}$ and \eqref{eq.jacobian.pm.diff.nodes.1} we have that for $\alpha = C_5\|G_{LB}-G\|$
		$$\left\|J_1(\tilde{a}-a )+J_2(\tilde{x}- \bar{x})\right\| \le  \alpha \|\tilde{a}-a\| .$$
		We now apply Proposition \ref{prop.matrix.diff} for:
		$$A=J_1,\tab B=J_2,\tab v=\tilde{a}-a,\tab u=\tilde{x}-\bar{x},\tab \alpha = C_5\|G_{LB}-G\|.$$
	    We get that
		\begin{align}\label{eq.lower.bound.nodes.main}
			\|\tilde{x}-\bar{x}\| & \ge  C_3 \e' h^{-2d+1} \left(
			\frac{1-C_5\|G_{LB}-G\|\;\|J^{-1}_1\|}{\|J^{-1}_1\|\;\|J_2\|}\right).
		\end{align}
		Define the constant
		\begin{equation}\label{eq.c5}
			C_6(d,\eta,m,M)= \min\left[\frac{1}{2C_8C_5C_4}, R\right].
		\end{equation}
		Then for $\e \le C_6 h^{2d-1}$ the numerator in \eqref{eq.lower.bound.nodes.main} satisfies
	 	\be\label{eq.dinominator.nodes.lower.bound}	1-C_5\|G_{LB}-G\|\;\|J^{-1}_1\| \ge \frac{1}{2}. \ee
	 	Where above we used Corollary \ref{cor:main6} to upper bound $\|G_{LB}-G\|$ by $C_4 \e' h^{-2d+1}$.
		By the previously derived bounds on $\|J^{-1}_1\|,\;\|J_2\|$ and by inequality \eqref{eq.dinominator.nodes.lower.bound}
		\be\label{eq.lower.bound.nodes.int}
			 \left(
			\frac{1-C_5\|G_{LB}-G\|\;\|J^{-1}_1\|}{\|J^{-1}_1\|\;\|J_2\|}\right) \ge \frac{1}{2\|J^{-1}_1\|\;\|J_2\|} \ge \frac{1}{2C_8\;d(d-1)M}.
		\ee
		
		Plugging \eqref{eq.lower.bound.nodes.int} back into \eqref{eq.lower.bound.nodes.main} we have that for
		$\e \le C_6
		h^{2d-1}$
		\begin{equation}
			\|\tilde{x}-\bar{x}\| \ge
			C_3 \e' h^{-2d+1}\left( \frac{1}{2 C_8\;d(d-1)M}\right).
		\end{equation}
		Fixing
		\begin{equation}\label{eq.k3}
			K_3(d,\eta,m,M) = \frac{C_3 }{2 C_8\;d(d-1)M},
		\end{equation}
		we get that
		$K_3 \left((1+|\kappa|)h\right)^{-2d+1} \e \le \|\tilde{x}-\bar{x}\|$. This concludes the proof of Proposition \ref{prop.lower.bound.nodes}.
	\end{proof}
	
	\smallskip
	
	We now prove the lower bound for the worst case error of the amplitudes of $G$.
	
	\begin{proposition}\label{prop.lower.bound.amplitudes}
	For $G_{LB}=(\tilde{a},\tilde{x})$, $G=(a,\bar{x})$ as above, and for $\e\le C_7 h^{2d-1}$, it holds that
	$$\|\tilde{a}-a\| \ge K_4 \left(\frac {1}{(1+|\kappa|)h}\right)^{2d-1} \e ,$$
    where $K_4, C_7$ are constants depending only on $d, \eta, m, M,$ defined within the proof.
	\end{proposition}
	
	\begin{proof}
		The proof for Proposition \ref{prop.lower.bound.amplitudes} goes along similar lines as that of Proposition \ref{prop.lower.bound.nodes}.
		Consider the following blocks of the Jacobian matrix at $G$ given in equation \eqref{eq.model.jacobian}.
		Let $J_3=J(1:d,1:d)$ and $J_4=J(1:d,d+1:2d)$. Let
		$P_{1,d}:
		\mathbb{R}^{2d} \rightarrow \mathbb{R}^{d}$ be the projection to the coordinates $(2,\ldots,d+1)$,
		 i.e. for $v=(v_0,\ldots,v_{2d-1}) \in \mathbb{R}^{2d}$, $P_{1,d}v =P_{1,d}(v) = (v_1,\ldots,v_{d})$.
		By Proposition \ref{prop.jacobian.PM.diff} we get that
		\begin{align}\label{eq.jacobian.pm.diff.nodes.2}
			\begin{split}
					C_5\|G_{LB}-G\|^2 &\ge \left\|\big(PM(G_{LB})-PM(G)\big) -J \cdot \big(G_{LB}-G \big)\right\| \\
					&\ge \left\|P_{1,d}\bigg(\big(PM(G_{LB})-PM(G)\big) -J \cdot \big(G_{LB}-G \big)\bigg)\right\|\\
					 &= \left\|J_3(\tilde{a}-a )+J_4(\tilde{x}- \bar{x})\right\|.
			\end{split}
		\end{align}
		
		The block $J_4$ admits the following factorization
		\begin{equation}\label{eq.J4.factorization}
					J_4= diag(1,2,\ldots,d)V_d(\bar{x}_1,\ldots,\bar{x}_d)diag({a_1},\ldots,{a_d}),
		\end{equation}
		where $diag(1,2,\ldots,d)$ is the diagonal matrix with $(1,\ldots,d)$
		on the diagonal, $V_d(\bar{x}_1,\allowbreak\ldots,\bar{x}_d)$
		is the Vandermonde matrix over the nodes $(\bar{x}_1,\ldots,\bar{x}_d)$ and  $diag({a_1},\ldots,{a_d})$ is a
		diagonal matrix with the amplitudes on the diagonal.
		\smallskip
		
		By Theorem \ref{thm.GAUTSCHI.1} and the factorization given in equation \eqref{eq.J4.factorization} we have that
		$\|J_4^{-1}\| \le \max[ \allowbreak 1,m^{-1}] C_8$.
		We also have that
		$\|J_3\| \le d$.
		
		\smallskip
		
		By equation \eqref{lower.bound.G} $\|G_{LB}-G\| \ge C_3  h^{-2d+1} \e'$,
		$\e'=(1+|\kappa|)^{-2d+1} \e$.
		Hence, either $\|\tilde{a}-a\| = \|G_{LB}-G\| \ge  C_3  h^{-2d+1} \e'$, and
		in this case setting $K_4 = C_3$ and $C_7=R$ we are done.
		Else,
		\be\label{eq.latter.case.2}
			\|\tilde{x}-\bar{x}\| = \|G_{LB}-G\|.
		\ee
		
		\smallskip
			
		We continue under the assumption of equation $\eqref{eq.latter.case.2}$.
		From $\eqref{eq.latter.case.2}$ and \eqref{eq.jacobian.pm.diff.nodes.2} we have that for $\alpha = C_5\|G_{LB}-G\|$
			$$\left\|J_3(\tilde{a}-a )+J_4(\tilde{x}- \bar{x})\right\| \le  \alpha \|\tilde{x}-\bar{x}\| .$$
		We now apply Proposition \ref{prop.matrix.diff} for:
		$$A=J_4,\tab B=J_3,\tab v=\tilde{x}-\bar{x},\tab u=\tilde{a}-a,\tab \alpha = C_5\|G_{LB}-G\|.$$
	    We get that
		\begin{align}\label{eq.lower.bound.amplitudes.main}
			\|\tilde{a}-a\| & \ge  C_3 \e' h^{-2d+1} \left(
			\frac{1-C_5\|G_{LB}-G\|\;\|J^{-1}_4\|}{\|J^{-1}_4\|\;\|J_3\|}\right).
		\end{align}
		Define the constant
		\begin{equation}\label{eq.c7}
			C_7(d,\eta,m,M)= \min\left[\frac{1}{2\max[1,m^{-1}]C_8C_5C_4}, R\right].
		\end{equation}
		Then for $\e \le C_7 h^{2d-1}$ 	
	 	\be\label{eq.dinominator.amplitudes.lower.bound}	1-C_5\|G_{LB}-G\|\;\|J^{-1}_4\| \ge \frac{1}{2}. \ee
	 	Where above we used Corollary \ref{cor:main6} to upper bound $\|G_{LB}-G\|$ by $C_4 h^{-2d+1} \e' $.
		By the previously derived bounds on $\|J^{-1}_4\|,\;\|J_3\|$ and by inequality \eqref{eq.dinominator.amplitudes.lower.bound}
		\be\label{eq.lower.bound.amplitudes.int}
			 \left(
			\frac{1-C_5\|G_{LB}-G\|\;\|J^{-1}_4\|}{\|J^{-1}_4\|\;\|J_3\|}\right) \ge \frac{1}{2\|J^{-1}_4\|\;\|J_3\|} \ge \frac{1}{2 \max[1,m^{-1}] C_8 \;d}.
		\ee
		Plugging \eqref{eq.lower.bound.amplitudes.int} back into \eqref{eq.lower.bound.amplitudes.main} we have that for
		$\e \le C_7
		h^{2d-1}$
		\begin{equation}
			\|\tilde{a}-a\| \ge
			C_3 \e' h^{-2d+1}\left( \frac{1}{2 \max[1,m^{-1}] C_8 \;d}\right).
		\end{equation}
		Fixing
		\begin{equation}\label{eq.k4}
			K_4 = \frac{C_3 }{2 \max[1,m^{-1}] C_8 \;d}
		\end{equation}
		we get that
		$K_4  \left((1+|\kappa|) h\right)^{-2d+1} \e \le \|\tilde{a}-a\|$. This concludes the proof of Proposition \ref{prop.lower.bound.amplitudes}.
%
%
%
%
%
%
			
	\end{proof}
	
	By Propositions \ref{prop.lower.bound.nodes} and \ref{prop.lower.bound.amplitudes}:
	\begin{itemize}
	  \item For $\e\le C_6 h^{2d-1}$, $\|\tilde{x}-\bar{x}\| \ge K_3  \left(\frac {1}{(1+|\kappa|)h}\right)^{2d-1} \e$.
	  \item For $\e\le C_7 h^{2d-1}$, $\|\tilde{a}-a\| \ge K_4 \left(\frac {1}{(1+|\kappa|)h}\right)^{2d-1} \e$.
	\end{itemize}
	To complete the proof of Theorem \ref{thm.lower.bound} we now set $F=\Psi^{-1}_{\kappa,h}(G)$ and $F_{LB} =
	\Psi^{-1}_{\kappa,h}(G_{LB}) \in E_{\e}(F)$.
	In this transformation the amplitudes $\tilde{a},a$ remain unchanged, while the nodes $\bar{x}, \tilde{x}$ are
	multiplied by $h$ (and shifted by $\kappa$). Hence, denoting $F_{LB}=(\tilde{a},\hat{x})$:
	$$
		\|\hat{x}-x\| \ge K_3  \left(\frac {1}{(1+|\kappa|)h}\right)^{2d-1}h \e,
	$$
	$$
		\|\tilde{a}-a\|,\;\|F_{LB}-F\| \ge K_4  \left(\frac {1}{(1+|\kappa|)h}\right)^{2d-1} \e.
	$$
	This proves the stated lower bounds of Theorem \ref{thm.lower.bound}.
\end{proof}

\smallskip

Till now we have assumed that all the $d$ nodes of the signal $F$ form a cluster of size $h$.
The {\it lower bounds} of Theorem \ref{thm.lower.bound} can be easily extended to the case where there are also non-cluster nodes:

\bc\label{cor:main6} Let $F\in {\cal P}_d$. Assume that some $s\le d$ of the nodes of $F$ form
an $(h,\kappa,\eta,m,M)$-regular cluster then:
\begin{itemize}
  \item For each positive $\e\le C_6 h^{2s-1}$
		$$K_3 \left(\frac {1}{(1+|\kappa|)h}\right)^{2s-1} h \e  \le \rho_x(F,\e).$$
  \item For each positive $\e\le C_7 h^{2s-1}$
		$$K_4 \left(\frac{1}{(1+|\kappa|)h}\right)^{2s-1} \e \le \rho(F,\e),\;\rho_a(F,\e).$$  		
\end{itemize}
The constants $K_3,\; K_4, \; C_6,\; C_7$ are the same constants as in Theorem \ref{thm.lower.bound} but with $d$ replaced with $s$.
\ec
\begin{proof}
	The required lower bounds follows directly from Theorem \ref{thm.lower.bound}.
	Indeed, we can perturb only the nodes and the amplitudes in the cluster,
	leaving the other nodes and amplitudes fixed,
	and then all the calculations and estimates above remain unchanged.
\end{proof}

\smallskip

\begin{remark}
	In the presence of non-cluster nodes obtaining the {\it upper bounds}
	for the worst case reconstruction error requires additional considerations.
	Indeed, perturbing both the cluster and the non-cluster nodes and the amplitudes
	a priori may create even larger deviations than those of Theorem \ref{thm.upper.bound}, with the moments,
	remaining within $\e$ of the original ones. Accuracy estimates in
	this situation presumably require analysis of several geometric scales at once.
	There are important open questions related to this multi-scale analysis.
	In particular, the following question was suggested in \cite{Can2014priv}:
	is it true (as numerical experiments suggest) that for well-separated non-cluster nodes,
	the accuracy of their reconstruction in Prony inversion is of order $\e$,
	independently of the size and structure of the cluster?
\end{remark} 

\smallskip

Our next result concerns the worst case accuracy of reconstruction of the Prony varieties $S_q(F)$.
The point is that the smaller is $q$ the larger is the variety $S_q(F)$,
but the higher is the accuracy of its reconstruction.
This fact was used in Section \ref{sec:improving.accuracy} in order to improve the reconstruction accuracy of the signal $F$ itself.
We will state this result only in the normalized signal space $\bar {\cal P}$.

\smallskip

Let $F$ form an $(h,\kappa,\eta,m,M)$-regular cluster and let $G$ be the model signal of $F$.
Recall the Hausdorff distance $d_H$ associated with the maximum metric:
for $A,B \subseteq {\cal P}$
\begin{align*}
		d_H(A,B) = \max\{\sup_{G'' \in A} \inf_{G''' \in B} \|G''-G'''\|,
		\sup_{G'' \in B} \inf_{G''' \in A} \|G''-G'''\|\}.
\end{align*}

Consider the local Prony variety $S^{\pi}_{q,\e}(G)=S_{q}(G)\cap \Pi_{\e,\frac{1}{h'}}(G)$, $h'=\frac{h}{1+|\kappa|}$,
and its possible reconstructions $S^{\pi}_{q,\e}(G') = S_{q}(G')\cap \Pi_{\e,\frac{1}{h'}}(G)$, $G' \in
\bar{E}_{\e}(F)$.
Define the worst case error in reconstruction of the local Prony
variety $S^{\pi}_{q,\e}(G)$ via the Hausdorff distance $d_H$:
$$
	\bar{\rho}_q(F,\e) = \max_{G'\in \bar{E}_{\e}(F)}d_H\left(S^{\pi}_{q,\e}(G),S^{\pi}_{q,\e}(G')\right).
$$

\bt\label{thm.prony.reconstruction}
Let $F\in {\cal P}_d$ form an $(h,\kappa,\eta,m,M)$-regular cluster.
Set $\e'=(1+|\kappa|)^{-2d+1} \e$ and $h'=\frac{h}{1+|\kappa|}$.
Then for each positive $\e\le h'^{2d-1}R$
$$
	C_3 \e' \left(\frac{1}{h}\right)^{q} \le \bar{\rho}_q(F,\e) \le C_4 \e
	\left(\frac{1}{h'}\right)^{q}, $$
where $C_3, C_4, R$ are the constants defined in Theorem \ref{thm:inv.fn}.
\et
\begin{proof}
	Define the Hausdorff distance $d_{H}^{\cal M}$ associated with the moment metric $d$:\\
	For $A,B \subseteq {\cal P}$,
	\begin{align*}
		d_{H}^{\cal M}(A,B) = \max\{\sup_{G'' \in A} \inf_{G''' \in B} d(G'',G'''),
		\sup_{G'' \in B} \inf_{G''' \in A} d(G'',G''')\}.
	\end{align*}
	Let $G$ be the model signal of $F$.
	Define the worst case error in reconstruction of the local Prony
	variety $S^{\pi}_{q,\e}(G)$, in the moment metric, by
	\begin{align*}
		\tilde{\rho}_q(F,\e) =  \max_{G'\in \bar{E}_{\e}(F)}d_{H}^{\cal M}\left(S^{\pi}_{q,\e}(G),S^{\pi}_{q,\e}(G')\right).
	\end{align*}
	
	For each $G'\in \bar{E}_{\e}(F)$, the Prony varieties $S_q(G), S_q(G')\subset \bar {\cal P}$ are the
	moment coordinate subspaces given by
 	\begin{align*}
 			S_q(G) &= \{ G'': m_k(G'')= m_k(G), \ k=0,\ldots,q \},\\
 			S_q(G') &= \{ G'': m_k(G'')= m_k(G'), \ k=0,\ldots,q\}.
 	\end{align*}
	The Hausdorff distance between them, with respect to the moment metric $d$,
	is equal to $\max_{k=0,\ldots,q} \allowbreak |m_k(G)-m_k(G')|.$
	As a result, for every $\e>0$,
	$$ \tilde{\rho}_q(F,\e) = \max_{G' \in \bar E_\e(F)} \;\max_{k=0,\ldots,q} |m_k(G)-m_k(G')|.$$
	By the first statement of Theorem \ref{thm.error.set.geometry.shift},
	$$
			 \Pi_{\e',\frac{1}{h}}(G) \subset \bar E_\e(F) \subset \Pi_{\e,\frac{1}{h'}}(G).	
	$$
	Therefore, for every $\e>0$, 	
	\be\label{eq.prony.recon.moment.matrix}
			 \e' \left(\frac{1}{h}\right)^q \le \tilde{\rho}_q(F,\e) \le  \e \left(\frac{1}{h'}\right)^{q}.	
	\ee
	For  $\e\le h'^{2d-1}R$,
	$$PM(S^{\pi}_{q,\e}(G)), PM(S^{\pi}_{q,\e}(G')) \subset PM( \Pi_{\e,\frac{1}{h'}}(G)) \subset
	Q_R(PM(G)).$$ We can therefore apply the equivalence of the moment and the maximum metrics given in Corollary
	\ref{cor:coord.moments} and get, from equation \eqref{eq.prony.recon.moment.matrix}, the required result of Theorem
	\ref{thm.prony.reconstruction}.
\end{proof}

\smallskip

Notice that, essentially, Theorems \ref{thm.upper.bound} and \ref{thm.lower.bound}
\footnote{
		Theorem \ref{thm.upper.bound}, stated in the original signal space ${\cal P}$, is
		strictly a special case of the upper bound given in Theorem \ref{thm.prony.reconstruction}, stated in the model space ${\cal \bar{P}}$.
		Theorem \ref{thm.lower.bound} and the lower bound given in Theorem \ref{thm.prony.reconstruction} has the same asymptotic in $h$.
		However, the constants and the required size of $\e$
		are different as in the case of the lower bound in the original space $\cal P$, we need to ensure that the projection
		of the error into amplitude space is non degenerate.
	}
are a special case of Theorem \ref{thm.prony.reconstruction}, for $q=2d-1$,
besides the separate bounds for the amplitudes and the nodes in Theorems \ref{thm.upper.bound} and \ref{thm.lower.bound}, which we do not address in
Theorem \ref{thm.prony.reconstruction}.

\smallskip

We conclude this section with the following corollary
which justifies the non-linear reconstruction procedure we described in Section \ref{sec:improving.accuracy}.

\bc\label{cor.upper.bound.prony.varieties.reconstruction}
	Let $F=(a,x)\in {\cal P}_d$ form an $(h,\kappa,\eta,m,M)$-regular cluster,
	let $\e \le R {h'}^{2d-1}$ and set $h'=\frac{h}{1+|\kappa|}$.
	Then for any $F'\in E_{\e}(F)$ and for any $q=0,\ldots,2d-1$:
	\begin{enumerate}
	  \item $F$ is contained within the $\Delta'_q$-neighborhood of the Prony variety $ S_{q}(F')$, for
   		$$
 			\Delta'_q=C_4\left(\frac{1}{h'}\right)^{q}\e.
 		$$
 	  \item The nodes vector $x$ is contained within
 	   an $h\Delta'_{q}$-neighborhood of the projection of the Prony variety $S_{q}(F')$ into the nodes coordinates, $S_{q}^x(F')$.
	\end{enumerate}
	The constants $R,C_4$ are as defined in Theorem \ref{thm:inv.fn}.
\ec
\begin{proof}
	First we show certain invariance of the Prony varieties under shift and scale transformations.
	\begin{proposition}\label{prop.sqg.sqf}
		Let $F\in {\cal P}$ and for $h>0, \kappa \in \mathbb{R}$, let  $G=\Psi_{\kappa,h}(F)$.
		Then, for each $q=0,\ldots,2d-1$, the Prony varieties $S_q(F)$ and $S_q(G)$ satisfies
		$$S_q(F) = \Psi_{\kappa,h}^{-1}\left(S_q(G)\right).$$
	\end{proposition}
	The above is simply a result of both the shift and the scale transformations, on the moments space, being
	triangular.
	Formally:
	\begin{proof}
		Let $F' \in \Psi^{-1}_{\kappa,h}(S_{q}(G))$ and let $G' \in S_q(G)$ be the signal such $\Psi_{\kappa,h}(F')=SC_{\frac{1}{h}}SH_{\kappa}(F')=G'$.
		By equation \ref{eq.scalling} and Proposition \ref{prop:shift3}, the moments of $F'$ are expressed via the moments of $G'$ as
		$$m_{k}(F') = \sum_{l=0}^k \binom{k}{l}(\kappa)^{k-l} h^{l}m_l(G').$$
		By definition of $S_{q}(G)$, for all $k=0,\ldots,q$ and for all $G' \in S_{q}(G)$, $m_k(G') = m_{k}(G)$. Therefore,
		for all $k=0,\ldots,q$, $m_{k}(F') = \sum_{l=0}^k \binom{k}{l}(\kappa)^{k-l} h^{l}m_l(G)=m_{k}(F)$ and hence $F'\in S_{q}(F)$. Since each step
		is reversible this completes the proof.
	\end{proof}
	
	Now the proof of Corollary \ref{cor.upper.bound.prony.varieties.reconstruction} follows directly from combining the upper bound given in Theorem
	\ref{thm.prony.reconstruction} and Proposition \ref{prop.sqg.sqf}.
\end{proof}

\appendix
	\section{Fourier and moment error sets equivalence}\label{appendix.moment.fourier}

	\begin{proof}[Proof of Corollary \ref{cor.fourier.error.set}]
		Let $F=(a,x) \in {\cal P}_d$ form an $(h,\kappa,\eta,m,M)$-regular cluster. 
		Let $\O$ be such that $\srf=(\O h)^{-1} \ge \Cr{srf.low.cor.main}$, and let $\e$ be such that $0 \le \epsilon\le \Cr{noise.cor.main} (\srf)^{-2d+1}$, where
		$\Cr{srf.low.cor.main}, \Cr{noise.cor.main}$ will be specified within the proof.
		
		Put $F^{({\Omega})} =
		\Psi_{\kappa,\frac{1}{\Omega}}(F)$.
		The signal $F^{({\Omega})}$ forms an $(\Omega h,0,\eta,m,M)$-regular cluster.
		Put $\e'=\Cr{high.equi} \e$, where $\Cr{high.equi}$ is the constant from the moment-Fourier 
		error sets equivalence relation \eqref{eq.fourier.moment.error.set.relation.1}. We invoke Theorem 
		\ref{thm.error.set.geometry.no.shift} and get that for 
		\be\label{eq.apen.fourier.main.3}\Cr{high.equi}\e=\e' \le R(\O h)^{2d-1},\ee 
		$\bar E_{\e'}(F^{(\O)})$ is contained
		within the $C_4\left(\frac{1}{\O h}\right)^{q}{\e'}$-neighborhood of the part of the Prony variety 
		$$S_{q,{\e'},\frac{1}{\O h}}\left(\Psi_{0,\O h}(F^{(\O)})\right) \subset S_{q}\left(\Psi_{0,\O h}(F^{(\O)})\right),$$
		(see Definition \ref{def.part.prony.ver}).
		By Proposition \ref{prop.sqg.sqf} 
		\begin{equation}\label{eq.apen.fourier.main.1} 
			S_{q}\left(\Psi_{0,\O h}(F^{(\O)})\right) =
			\Psi_{0,\O h}\left(S_{q}(F^{(\O)})\right),
		\end{equation}
		then with \eqref{eq.apen.fourier.main.1} and the above, $\Psi_{0,\O
		h}\left(E_{\e'}(F^{(\O)})\right)=\bar E_{\e'}(F^{(\O)})$ is
		contained within the 
		$$C_4\left(\frac{1}{\O h}\right)^{q}{\e'}\mbox{-neighborhood}$$ 
		of $\Psi_{0,\O h}\left(S_{q}(F^{(\O)})\right)$. Therefore applying $\Psi_{0,\O
		h}^{-1}$ on both $\Psi_{0,\O
		h}\left(E_{\e'}(F^{(\O)})\right)$ and $\Psi_{0,\O h}\left(S_{q}(F^{(\O)})\right)$, we get that  
		$E_{\e'}(F^{(\O)})$ is contained within the 
		$$C_4\left(\frac{1}{\O h}\right)^{q}{\e'}\mbox{-neighborhood}$$ 
		of $S_{q}(F^{(\O)})$.
		Now we invoke relation \eqref{eq.fourier.moment.error.set.relation.1} and get that 
		$$\Psi_{\kappa,\frac{1}{\Omega}}\big(E_{\epsilon,\Omega}(F)\big)\subseteq E_{\Cr{high.equi}
		\e}(F^{(\O)})=E_{\e'}(F^{(\O)}),$$
		provided that 
		\be\label{eq.apen.fourier.main.cond.2} \srf \ge \Cr{srf.low.equi} \mbox{\tab and \tab } 0 \le \epsilon\le \Cr{noise.up.equi} (\srf)^{-2d+1}.\ee
		Consequently $\Psi_{\kappa,\frac{1}{\Omega}}\big(E_{\epsilon,\Omega}(F)\big)$ is contained within the
		$C_4\left(\frac{1}{\O h}\right)^{q}{\e'}$-neighborhood of $S_{q}(F^{(\O)})$. Again by Proposition \ref{prop.sqg.sqf}
		$$S_{q}(F^{(\O)})=S_{q}\left(\Psi_{\kappa,\frac{1}{\Omega}}(F)\right)
		= \Psi_{\kappa,\frac{1}{\Omega}}\left(S_{q}(F)\right),$$
		therefore $E_{\epsilon,\Omega}(F)$ is contained within the $C_4\left(\frac{1}{\O
		h}\right)^{q}{\e'}$-neighborhood of $S_{q}(F)$. Replacing $\e'$ with $\Cr{high.equi} \e$ we get that 
		$E_{\epsilon,\Omega}(F)$ is contained within the $\Cr{high.equi} C_4 \left(\frac{1}{\O
		h}\right)^{q}{\e}$-neighborhood of $S_{q}(F)$. 
		
		Set $\Cr{cor.main}=\Cr{high.equi}C_4$, then we have proved 
		that $E_{\epsilon,\Omega}(F)$ is contained within the $\Cr{cor.main}\left(\frac{1}{\O
		h}\right)^{q}{\e}$-neighborhood of $S_{q}(F)$, under the conditions \eqref{eq.apen.fourier.main.3}
		and \eqref{eq.apen.fourier.main.cond.2} on $\e$ and $\srf$. Finally setting $\Cr{srf.low.cor.main}=\Cr{srf.low.equi}$ and
		$\Cr{noise.cor.main}=\min(\Cr{noise.up.equi},\frac{R}{\Cr{high.equi}})$ ensures that these conditions hold.   
	\end{proof}
	\begin{proof}[Proof of Corollary \ref{cor.fourier.worst}]
		Let $F=(a,x) \in {\cal P}_d$ form an $(h,\kappa,\eta,m,M)$-regular cluster. 
		Let $\O$ be such that $\srf=(\O h)^{-1} \ge \Cr{cor.worst.srf}$, and let $\e$ be such that $0 \le \epsilon\le \Cr{cor.worst.noise} (\srf)^{-2d+1}$,
		where $\Cr{cor.worst.srf}, \Cr{cor.worst.noise}$ will be specified within the proof. 
		
		Put $F^{({\Omega})} = (a,\bar{x})=
		\Psi_{\kappa,\frac{1}{\Omega}}(F)$
		and note that the signal $F^{({\Omega})}$ forms an $(\Omega h,0,\eta,m,M)$-regular cluster.
		
		Put $\e'=\Cr{high.equi} \e$, where $\Cr{high.equi}$ is the constant from the moment-Fourier 
		error sets equivalence relation \eqref{eq.fourier.moment.error.set.relation.1}.
		By Theorem \ref{thm.upper.bound} (moment reconstruction upper bounds),
		\begin{align}\label{eq.apen.fourier.upper.0}
			\begin{split}
    			\max_{F'\in E_{\e'}(F^{(\O)})}\|F'-F^{(\O)}\| & \le C_4 \srf^{2d-1}\e',\\
    			\max_{F'=(a',x')\in E_{\e'}(F^{(\O)})}\|a'-a\| & \le C_4 \srf^{2d-1}\e',\\
    			\max_{F'=(a',x')\in E_{\e'}(F^{(\O)})}\|x'-\bar{x}\| & \le C_4 \srf^{2d-2}\e',
    		\end{split}  
    	\end{align}
    	provided that 
    	\be\label{eq.apen.fourier.upper.cond.0} \e' \le \srf^{-2d+1} R.\ee
    	We apply relation \eqref{eq.fourier.moment.error.set.relation.1} and get that
    	\be\label{eq.apen.fourier.upper.1} 
    		\Psi_{\kappa,\frac{1}{\Omega}}\big(E_{\epsilon,\Omega}(F)\big)\subseteq
    		 E_{\Cr{high.equi} \e}(F^{(\O)})=E_{\e'}(F^{(\O)}),
		\ee 
		provided that 
		\be\label{eq.apen.fourier.upper.cond.1} 
			\srf \ge \Cr{srf.low.equi} \mbox{\tab and \tab } 0 \le \epsilon\le \Cr{noise.up.equi} \srf^{-2d+1}.
		\ee
		Now by combining \eqref{eq.apen.fourier.upper.0} and \eqref{eq.apen.fourier.upper.1} we get that 
		for $\e$ and $\srf$ satisfying \eqref{eq.apen.fourier.upper.cond.0} and \eqref{eq.apen.fourier.upper.cond.1},
		it holds that
		\begin{align}\label{eq.apen.fourier.upper.2'}
			\begin{split}
    			\max_{F'\in \Psi_{\kappa,\frac{1}{\Omega}}\big(E_{\epsilon,\Omega}(F)\big)}
    			   \|F'-F^{(\O)}\| & \le C_4 \srf^{2d-1}\e',\\
    			\max_{F'=(a',x')\in \Psi_{\kappa,\frac{1}{\Omega}}\big(E_{\epsilon,\Omega}(F)\big)}
    			   \|a'-a\| & \le C_4 \srf^{2d-1}\e',\\
    			\max_{F'=(a',x')\in \Psi_{\kappa,\frac{1}{\Omega}}\big(E_{\epsilon,\Omega}(F)\big)}
    			   \|x'-\bar{x}\| & \le C_4 \srf^{2d-1} \e'.
    		\end{split}  
    	\end{align}  
    	Consequently applying $\Psi^{-1}_{\kappa,\frac{1}{\Omega}}$ and replacing $\e'$ with $\Cr{high.equi} \e$,
    	we get that for $\e$ and $\srf$ satisfying \eqref{eq.apen.fourier.upper.cond.0} and
    	\eqref{eq.apen.fourier.upper.cond.1}
		\begin{align}\label{eq.apen.fourier.upper.2}
			\begin{split}
    			\max_{F'\in E_{\e,\O}(F)}\|F'-F\| & \le C_4 \srf^{2d-1}\e'=\Cr{high.equi}C_4 \srf^{2d-1}\e,\\
    			\max_{F'=(a',x')\in E_{\e,\O}(F)}\|a'-a\| & \le C_4 \srf^{2d-1}\e'=\Cr{high.equi}C_4\srf^{2d-1}\e,\\
    			\max_{F'=(a',x')\in E_{\e,\O}(F)}\|x'-x\| & \le C_4 \srf^{2d-1} h\e'=\Cr{high.equi}C_4 \srf^{2d-1} h\e.
    		\end{split}  
    	\end{align}  
    	Simplifying conditions \eqref{eq.apen.fourier.upper.cond.0} and \eqref{eq.apen.fourier.upper.cond.1},
    	we get that \eqref{eq.apen.fourier.upper.2} holds for
    	\be\label{eq.apen.fourier.upper.cond.final}	
    		\srf \ge \Cr{srf.low.equi} \tab \mbox{and}\tab \e \le
    		\Cl[B]{cor.worst.fourier.upper.noise.final}\srf^{-2d+1},
    	\ee 
    	where $\Cr{cor.worst.fourier.upper.noise.final}=\min\left(\frac{R}{\Cr{high.equi}},\Cr{noise.up.equi}\right)$.
    	This proves the upper bounds of Corollary \ref{cor.fourier.worst} for any $\Cr{cor.worst.srf}\ge \Cr{srf.low.equi}$
    	and $\Cr{cor.worst.noise}\le \Cr{cor.worst.fourier.upper.noise.final}$.
    	
    	We will now show the lower bounds. Put $\e'=\Cr{low.equi}\e$,
    	where $\Cr{low.equi}$ is the constant from the moment-Fourier 
		error sets equivalence relation \eqref{eq.fourier.moment.error.set.relation.1}. By Theorem \ref{thm.lower.bound}
		(moment reconstruction lower bounds),
		\begin{align}\label{eq.apen.fourier.lower.0}
			\begin{split}
    			\max_{F'\in E_{\e'}(F^{(\O)})}\|F'-F^{(\O)}\| & \ge K_4 \srf^{2d-1}\e',\\
    			\max_{F'=(a',x')\in E_{\e'}(F^{(\O)})}\|a'-a\| & \ge K_4 \srf^{2d-1}\e',\\
    			\max_{F'=(a',x')\in E_{\e'}(F^{(\O)})}\|x'-\bar{x}\| & \ge K_3 \srf^{2d-2}\e',
    		\end{split}  
    	\end{align}  
    	provided that 
    	\be\label{eq.apen.fourier.lower.cond.0} 
    		\e' \le \min(C_6,C_7)\srf^{-2d+1}.
    	\ee
    	We apply relation \eqref{eq.fourier.moment.error.set.relation.1} and get that
    	\be\label{eq.apen.fourier.lower.1} 
 			 E_{\e'}(F^{(\O)})=
 			 E_{\Cr{low.equi}\epsilon}(F^{(\O)}) \subseteq
		     \Psi_{\kappa,\frac{1}{\Omega}}\big(E_{\epsilon,\Omega}(F)\big),
		\ee 
		provided that 
		\be\label{eq.apen.fourier.lower.cond.1} 
			\srf \ge \Cr{srf.low.equi} \mbox{\tab and \tab } 0 \le \epsilon\le \Cr{noise.up.equi} \srf^{-2d+1}.
		\ee
		Now by combining \eqref{eq.apen.fourier.lower.0} and \eqref{eq.apen.fourier.lower.1} we get that 
		for $\e$ and $\srf$ satisfying \eqref{eq.apen.fourier.lower.cond.0} and \eqref{eq.apen.fourier.lower.cond.1},
		it holds that
		\begin{align}\label{eq.apen.fourier.lower.2}
			\begin{split}
				\max_{F'\in E_{\e,\O}(F)}\|F'-F\| & \ge K_4 \srf^{2d-1}\e'=B_3K_4 \srf^{2d-1}\e,\\
    			\max_{F'=(a',x')\in E_{\e,\O}(F)}\|a'-a\| & \ge K_4 \srf^{2d-1}\e'=B_3K_4 \srf^{2d-1}\e,\\
    			\max_{F'=(a',x')\in E_{\e,\O}(F)}\|x'-x\| & \ge K_3 \srf^{2d-1}h\e'=B_3K_3 \srf^{2d-1}h\e.
    		\end{split}  
    	\end{align} 
    	Simplifying conditions \eqref{eq.apen.fourier.lower.cond.0} and \eqref{eq.apen.fourier.lower.cond.1},
    	we get that \eqref{eq.apen.fourier.lower.2} holds for 
    	\be\label{eq.apen.fourier.lower.cond.final}	
    		\srf \ge \Cr{srf.low.equi} \mbox{\tab and \tab } 0 \le \epsilon\le
    		\Cl[B]{cor.worst.fourier.lower.noise} \srf^{-2d+1},
    	\ee
    	where $\Cr{cor.worst.fourier.lower.noise}=\min\left(\Cr{low.equi}^{-1}\min(C_6,C_7),\Cr{noise.up.equi}\right)$.
    	
    	\smallskip
    	
    	Finally combining the conditions \eqref{eq.apen.fourier.upper.cond.final} where the upper bounds of
    	\eqref{eq.apen.fourier.upper.2} hold, and the conditions \eqref{eq.apen.fourier.lower.cond.final} where the lower
    	bounds of \eqref{eq.apen.fourier.lower.2} hold, we get that for $\Cr{cor.worst.srf}=\Cr{srf.low.equi}$ and 
    	$\Cr{cor.worst.noise}=\min\left(\Cr{cor.worst.fourier.upper.noise.final},\Cr{cor.worst.fourier.lower.noise}\right)$,
    	both \eqref{eq.apen.fourier.upper.2} and  \eqref{eq.apen.fourier.lower.2} hold thus proving Corollary
    	\ref{cor.fourier.worst}.
	\end{proof}  
	
 	\section{Quantitative inverse function theorem}	\label{appendix.Quantitative}
 		Let $G$ be an $(\eta,m,M)$-regular signal and $PM(G)=\nu=(\nu_1,\ldots,\nu_{2d-1})$.
		To prove Theorem \ref{thm:inv.fn} statement 2 we need to explicitly give constants $R,C_3,C_4$ depending only on $d,\eta,m,M$
		such that: The inverse mapping $PM^{-1}$ is regular analytic in the cube $Q_R(\nu)$
		and for each $\nu',\nu'' \in Q_R(\nu)$
		$$
		C_3\|\nu''-\nu'\|\le \|PM^{-1}(\nu'')-PM^{-1}(\nu')\|\le C_4\|\nu''-\nu'\|.
		$$
 		\begin{ThmInvStatement2}
 			Let $J=J(G)$ be the Jacobian matrix at $G$.
 			Let $C_1=C_1(m,\eta, d),\allowbreak \; C_2=C_2(d,M)$ be the the constants derived in statement 1 of Theorem \ref{thm:inv.fn} satisfying
 			$$\|J^{-1}\| \le C_1,\; \|J\|\le C_2.$$
 			Then for
 			\begin{align*}
				R =  \left( 48 \cdot C_1^2 \cdot d (M+1)(2d-1)^2\right)^{-1},\tab
				C_3 = \frac{2C_1}{1+2C_1C_2},\tab
				C_4 =  2C_1,	
			\end{align*}
 			the inverse mapping $PM^{-1}$ is regular analytic in the cube $Q_R(\nu)$
			and for each $\nu',\nu'' \in Q_R(\nu)$
			$$
			C_3\|\nu''-\nu'\|\le \|PM^{-1}(\nu'')-PM^{-1}(\nu')\|\le C_4\|\nu''-\nu'\|.
			$$
 		\end{ThmInvStatement2}

 		\begin{proof}[Proof Theorem \ref{thm:inv.fn}, statement 2]
 			The next proposition provides a Lipschitz constant for the difference between $PM$ and its linear part in the neighborhood of $G$.
 		\begin{proposition}\label{prop.jacobian.PM.diff.G'.G}
			Let $G$ be an $(\eta,m,M)$-regular signal. Let $r\le \frac{1}{2d-1}$ and $G'$ a signal such that
			$\|G'-G\|\le r$. Let $J=J(G)$ be the Jacobian
			matrix at $G$. Then
			$$\left|\left|\big(PM(G')-PM(G)\big) -J\cdot \big(G'-G \big)\right|\right| \le C_5(d,M)\cdot r \cdot\|G'-G\|,$$
			where $C_5=6(M+1)(2d-1)^2 d$.
		\end{proposition}
		\begin{proof}
			First for each $G''$ such that $\|G-G''\|\le r$ we have the the following upper bound
			on the second derivatives of the moments functions. For each moment of order $k=0,\ldots, 2d-1$,
			\begin{align}
				\left|\frac{\partial^2 m_k}{\partial x_i^2}(G'')\right|, \left|\frac{\partial^2 m_k }{\partial a_i\partial
				x_i}(G'')\right| \le 3(M+1)(2d-1)^2
			\end{align}
			while the rest of the second derivatives are zero.
			
			\smallskip
			
			Consider the standard multi-index notation.
			For $\alpha=(\alpha_1,..,\alpha_n),\allowbreak \alpha \in \{\mathbb{N} \cup 0\}^n$, we define:
			Absolute value, $|\alpha| =\alpha_1+..+\alpha_n$; Factorial, $\alpha ! = \alpha_1! \cdot \alpha_2! \cdots \alpha_n!$;
			Power, for $u \in \mathbb{R}^n$, $u^{\alpha} = u_1^{\alpha_1}\cdot .. \cdot u_n^{\alpha_n}$;
			Partial derivative, for $x=(x_1,\ldots,x_n) \in \mathbb{R}^n$,
			$D^{\alpha}=\frac{\partial^{|\alpha|}}{\partial x^\alpha} =
			\frac{\partial^{|\alpha|}}{\partial x_1^{\alpha_1}..\partial x_n^{\alpha_n}}$.
			
			\smallskip	
			
			Put $G'=G+h$ where $\|h\| \le r$. Taking the first order Taylor approximation with remainder we have that, for each
			$k=0,\ldots,2d-1$,
			\begin{align*}
				|(m_k(G')-m_k(G))-\nabla m_k(G)\cdot h| & =\left|\sum_{|\alpha|=2,\;\alpha \in \mathbb{N}^{2d}\cup
				\{0\}}\frac{1}{\alpha!}D^{\alpha}m_k(G_{k})h^{\alpha} \right|,
			\end{align*}
			where $G_k \in [G,G']$.
			\begin{align*}
				\left|\sum_{|\alpha|=2,\;\alpha \in \mathbb{N}^{2d}\cup
				\{0\}}\frac{1}{\alpha!}D^{\alpha}m_k(G_{k})h^{\alpha} \right|
				&\le \sum_{|\alpha|=2,\;\alpha \in \mathbb{N}^{2d}\cup
				\{0\}}|D^{\alpha}m_k(G_{k})|\cdot|h^{\alpha}| \\
				&\le r \|h\| \sum_{|\alpha|=2,\;\alpha \in \mathbb{N}^{2d}\cup
				\{0\}}|D^{\alpha}m_k(G_{k})|\\
					& \le r \|h\| 2d \big(3(M+1)(2d-1)^2\big)\\
					& =  6 (M+1)(2d-1)^2 d \cdot r \|h\|.
			\end{align*}
			The proposition follows.
		\end{proof}
 			
			\begin{corollary}\label{cor.jacobian.PM.diff.G'.G''}
				Let $G$ be an $(\eta,m,M)$-regular signal. Let $r	\le \frac{1}{2d-1}$ and let $G',G''$ be signals such that
				$G'',G' \in Q_r(G)$. Denote by $J=J(G)$ the Jacobian matrix at $G$. Then
				$$\left\|\bigg(PM(G'')-PM(G')\bigg) -J\cdot \bigg(G''-G' \bigg)\right\| \le 2 C_5(d,M)\cdot r \cdot\|G''-G'\|.$$
			\end{corollary}
			\begin{proof}
				It is a direct consequence of Proposition \ref{prop.jacobian.PM.diff.G'.G}.
			\end{proof}
				
			Fix $r=\frac{1}{4C_5C_1}$. Then for $G',G'' \in Q_r(G)$ we have, by Corollary
			\ref{cor.jacobian.PM.diff.G'.G''}, that:
			\begin{align*}
					\|PM(G'')-PM(G') \| &\ge \|J (G''-G')\| -2C_5r\|G''-G'\| \\
										   &=\|J (G''-G')\| -\frac{1}{2C_1}\|G''-G'\| \\
										   & \ge \frac{1}{C_1}\|G''-G'\| -\frac{1}{2C_1}\|G''-G'\|\\
										   &=\frac{1}{2C_1}\|G''-G'\|,	
			\end{align*}
			and
			\begin{align*}
					\|PM(G'')-PM(G')\| &\le \|J(G''-G')\| +2C_5r\|G''-G'\| \\
										   &=\|J (G''-G')\| +\frac{1}{2C_1}\|G''-G'\| \\
										   &\le C_2\|G''-G'\| +\frac{1}{2C_1}\|G''-G'\|\\
										   &=\left(\frac{1}{2C_1}+C_2 \right)\|G''-G'\|.	
			\end{align*}
			
			We conclude that for $r=\frac{1}{4C_5C_1}$ and $G',G'' \in Q_r(G)$, $PM$ is one to one
			on $Q_r(G)$ and satisfies there
			\begin{equation}\label{eq.apendix.1}
				\frac{1}{2C_1} \|G''-G'\| \le \|PM(G'')-PM(G')\|\le \left(\frac{1}{2C_1}+C_2 \right) \|G''-G'\|.
			\end{equation}
			
			Since $PM$ is one to one on the open cube $interior(Q_r(G))$,
			by invariance of domain theorem, $PM$ is a homeomorphism between $interior(Q_r(G))$ and $PM(\allowbreak
			interior(Q_r(G))$, and, $PM(interior\allowbreak(Q_r(G))$ is open.
			
			\smallskip

			Let $PM(G)=\nu$. By equation \eqref{eq.apendix.1}, we have that
			$PM(interior\allowbreak(Q_r(G))$ contains the cube of radius $R=\frac{1}{2C_1}r $, $Q_R(\nu)$,  and for each $\nu',\nu''
			\in Q_R(\nu)$
			$$
				\frac{2C_1}{1+2C_1C_2} \|\nu''-\nu'\| \le \|PM^{-1}(\nu'')-PM^{-1}(\nu')\| \le 2C_1 \|\nu''-\nu'\|.
			$$
			
			Fixing
			\begin{align}
				R = \frac{1}{2C_1}r, \ \ \ C_3 = \frac{2C_1}{1+2C_1C_2},\ \ \ C_4 =  2C_1,	
			\end{align}
			concludes the proof of statement 2 of Theorem \ref{thm:inv.fn}.
 		\end{proof}

\bibliographystyle{plain}
\bibliography{bib}{}

\begin{thebibliography}{10}

\bibitem{akinshin2015accuracy}
Andrey Akinshin, Dmitry Batenkov, and Yosef Yomdin.
\newblock Accuracy of spike-train {F}ourier reconstruction for colliding nodes.
\newblock In {\em 2015 International Conference on Sampling Theory and
  Applications (SampTA)}, pages 617--621. IEEE, 2015.

\bibitem{akinshin2017accuracy}
Andrey Akinshin, Gil Goldman, Vladimir Golubyatnikov, and Yosef Yomdin.
\newblock Accuracy of reconstruction of spike-trains with two near-colliding
  nodes.
\newblock In {\em Proc.\ Complex Analysis and Dynamical Systems VII}, volume
  699, pages 1--17. The AMS and Bar-Ilan University, 2015.

\bibitem{auton1981investigation}
Jon~R Auton and Michael~L Van~Blaricum.
\newblock Investigation of procedures for automatic resonance extraction from
  noisy transient electromagnetics data.
\newblock {\em Math. Notes}, 1:79, 1981.

\bibitem{batenkov2017accurate}
Dmitry Batenkov.
\newblock Accurate solution of near-colliding {P}rony systems via decimation
  and homotopy continuation.
\newblock {\em Theoretical Computer Science}, 2017.

\bibitem{batenkov2018conditioning}
Dmitry Batenkov, Laurent Demanet, Gil Goldman, and Yosef Yomdin.
\newblock Conditioning of partial nonuniform {F}ourier matrices with clustered
  nodes.
\newblock {\em arXiv preprint arXiv:1809.00658 [cs, math]}, 2018.

\bibitem{batenkov2019spectral}
Dmitry Batenkov, Benedikt Diederichs, Gil Goldman, and Yosef Yomdin.
\newblock The spectral properties of {V}andermonde matrices with clustered
  nodes.
\newblock {\em arXiv preprint arXiv:1909.01927}, 2019.

\bibitem{superres_clusters18}
Dmitry Batenkov, Gil Goldman, and Yosef Yomdin.
\newblock Super-resolution of near-colliding point sources.
\newblock {\em arXiv preprint arXiv:1904.09186}, 2019.

\bibitem{batenkov2013accuracy}
Dmitry Batenkov and Yosef Yomdin.
\newblock On the accuracy of solving confluent {P}rony systems.
\newblock {\em SIAM Journal on Applied Mathematics}, 73(1):134--154, 2013.

\bibitem{batenkov2013geometry}
Dmitry Batenkov and Yosef Yomdin.
\newblock Geometry and singularities of the {P}rony mapping.
\newblock In {\em Proceedings of 12th International Workshop on Real and
  Complex Singularities}, volume~10, pages 1--25, 2014.

\bibitem{Can2014priv}
Emmanuel~J. Cand{\`e}s.
\newblock private communication.
\newblock 2014.

\bibitem{demanet2015recoverability}
Laurent Demanet and Nam Nguyen.
\newblock The recoverability limit for superresolution via sparsity.
\newblock {\em arXiv preprint arXiv:1502.01385}, 2015.

\bibitem{donoho1992superresolution}
David~L Donoho.
\newblock Superresolution via sparsity constraints.
\newblock {\em SIAM journal on mathematical analysis}, 23(5):1309--1331, 1992.

\bibitem{fannjiang_compressive_2016}
Albert Fannjiang.
\newblock Compressive {{Spectral Estimation}} with {{Single}}-{{Snapshot
  ESPRIT}}: {{Stability}} and {{Resolution}}.
\newblock {\em arXiv:1607.01827 [cs, math]}, July 2016.

\bibitem{friedland2015doubling}
Omer Friedland and Yosef Yomdin.
\newblock Doubling coverings of algebraic hypersurfaces.
\newblock {\em arXiv preprint arXiv:1512.02903}, 2015.

\bibitem{gautschi1962inverses}
Walter Gautschi.
\newblock On inverses of {V}andermonde and confluent {V}andermonde matrices.
\newblock {\em Numerische Mathematik}, 4(1):117--123, 1962.

\bibitem{Prony.Sing.View.18}
Gil Goldman, Yehonatan Salman, and Yosef Yomdin.
\newblock Accuracy of noisy spike-train reconstruction: a singularity theory
  point of view.
\newblock {\em J. Singul.}, 18:409--426, 2018.

\bibitem{goldman2018geometry}
Gil Goldman, Yehonatan Salman, and Yosef Yomdin.
\newblock Geometry and singularities of {P}rony varieties.
\newblock {\em arXiv preprint arXiv:1806.02204}, 2018.

\bibitem{hubbard2015vector}
John~H Hubbard and Barbara~Burke Hubbard.
\newblock {\em Vector calculus, linear algebra, and differential forms: a
  unified approach}.
\newblock Matrix Editions, 5 edition, 2015.

\bibitem{kunis2017multidim.bounds}
Stefan Kunis, H.~Michael Moller, Thomas Peter, and Ulrich von~der Ohe.
\newblock {P}rony method under an almost sharp multivariate {I}ngham
  inequality.
\newblock {\em J. {F}ourier Anal. Appl.}, 24(5):1306--1318, 2018.

\bibitem{kunis2018condition}
Stefan Kunis and Dominik Nagel.
\newblock On the condition number of {V}andermonde matrices with pairs of
  nearly-colliding nodes.
\newblock {\em arXiv preprint arXiv:1812.08645}, 2018.

\bibitem{kunis2019smallest}
Stefan Kunis and Dominik Nagel.
\newblock On the smallest singular value of multivariate {V}andermonde matrices
  with clustered nodes.
\newblock {\em arXiv preprint arXiv:1907.07119}, 2019.

\bibitem{kunis2016multidim.Prony}
Stefan Kunis, Thomas Peter, Tim Romer, and Ulrich von~der Ohe.
\newblock A multivariate generalization of {P}rony's method.
\newblock {\em Linear Algebra Appl.}, 490:31--47, 2016.

\bibitem{li2017stable}
Weilin Li and Wenjing Liao.
\newblock Stable super-resolution limit and smallest singular value of
  restricted {F}ourier matrices.
\newblock {\em arXiv preprint arXiv:1709.03146}, 2017.

\bibitem{li2019super}
Weilin Li, Wenjing Liao, and Albert Fannjiang.
\newblock Super-resolution limit of the {ESPRIT} algorithm.
\newblock {\em arXiv preprint arXiv:1905.03782}, 2019.

\bibitem{liao2016music}
Wenjing Liao and Albert Fannjiang.
\newblock {MUSIC} for single-snapshot spectral estimation: Stability and
  super-resolution.
\newblock {\em Applied and Computational Harmonic Analysis}, 40(1):33--67,
  2016.

\bibitem{lindberg_mathematical_2012}
Jari Lindberg.
\newblock Mathematical concepts of optical superresolution.
\newblock {\em Journal of Optics}, 14(8):083001, 2012.

\bibitem{pereyra_exponential_2010}
Victor Pereyra and Godela Scherer.
\newblock {\em Exponential {{Data Fitting}} and {{Its Applications}}}.
\newblock {Bentham Science Publishers}, January 2010.

\bibitem{peter2013generalized}
Thomas Peter and Gerlind Plonka.
\newblock A generalized {P}rony method for reconstruction of sparse sums of
  eigenfunctions of linear operators.
\newblock {\em Inverse Problems}, 29(2):025001, 2013.

\bibitem{plonka2014prony}
Gerlind Plonka and Manfred Tasche.
\newblock {P}rony methods for recovery of structured functions.
\newblock {\em GAMM-Mitteilungen}, 37(2):239--258, 2014.

\bibitem{potts2010bounds}
Daniel Potts and Manfred Tasche.
\newblock Parameter estimation for exponential sums by approximate {P}rony
  method.
\newblock {\em Signal Processing}, 90(4):1631--1642, 2010.

\bibitem{prony1795essai}
R.~Prony.
\newblock Essai experimental et analytique.
\newblock {\em J. Ec. Polytech.(Paris)}, 2:24--76, 1795.

\bibitem{stoica_spectral_2005}
P.~Stoica and R.L. Moses.
\newblock {\em Spectral Analysis of Signals}.
\newblock {Pearson/Prentice Hall}, 2005.

\bibitem{vetterli2002sampling}
Martin Vetterli, Pina Marziliano, and Thierry Blu.
\newblock Sampling signals with finite rate of innovation.
\newblock {\em IEEE transactions on Signal Processing}, 50(6):1417--1428, 2002.

\end{thebibliography}

\end{document}